\documentclass[ejs,preprint]{imsart}

\RequirePackage[OT1]{fontenc}
\RequirePackage{amsthm,amsmath}
\RequirePackage[numbers]{natbib}
\RequirePackage[colorlinks,citecolor=blue,urlcolor=blue]{hyperref}
\usepackage{graphicx}

\usepackage{xifthen}

\newcommand{\domain}{D}

\newcommand{\fwd}[1][]{
    \ifthenelse{\isempty{#1}}
    {G}
    {G_{#1}}
}

\newcommand{\predpts}{\bm{X}}
\newcommand{\predptsprime}{\bm{X}'}

\newcommand{\datdim}{p}

\newcommand{\somestage}{i}

\newcommand{\stage}{n}

\newcommand{\currmean}[1][]{
   \ifthenelse{\isempty{#1}}
   {\mu_{\predpts}^{(\stage)}}
   {\mu_{#1}^{(\stage)}}
}
\newcommand{\prevmean}[1][]{
   \ifthenelse{\isempty{#1}}
   {\mu_{\predpts}^{(\stage-1)}}
   {\mu_{#1}^{(\stage-1)}}
}

\newcommand{\currcov}[1][]{
   \ifthenelse{\isempty{#1}}
   {K_{\predpts\predptsprime}^{(\stage)}}
   {K_{#1}^{(\stage)}}
}
\newcommand{\prevcov}[1][]{
   \ifthenelse{\isempty{#1}}
   {K_{\predpts\predptsprime}^{(\stage-1)}}
   {K_{#1}^{(\stage-1)}}
}

\newcommand{\gp}[1][]{
    \ifthenelse{\isempty{#1}}
    {Z}
    {Z_{#1}}
}

\newcommand{\covFunN}[1][]{
    \ifthenelse{\isempty{#1}}
    {p_{\Gamma}^{(\somestage)}}
    {p_{\Gamma}^{(\somestage)}({#1})}
}

\usepackage{hyperref}
\usepackage[nameinlink]{cleveref}
\ifpdf
  \DeclareGraphicsExtensions{.eps,.pdf,.png,.jpg}
\else
  \DeclareGraphicsExtensions{.eps}
\fi

\usepackage{bm}
\usepackage{amsmath}
\usepackage{amsthm}
\usepackage{amssymb}
\usepackage{subcaption}
\usepackage{color}

\usepackage{bbold}

\usepackage[inline]{enumitem}
\setlist[enumerate]{leftmargin=.5in}
\setlist[itemize]{leftmargin=.5in}

\theoremstyle{definition}
\newtheorem{definition}{Definition}

\newtheorem{lemma}{Lemma}
\newtheorem{theorem}{Theorem}
\newtheorem{corollary}{Corollary}

\newtheorem{remark}{Remark}

\newtheorem*{example*}{Example}
\crefname{hypothesis}{Hypothesis}{Hypotheses}

\newcommand\pcref[1]{(\Cref{#1})}

\usepackage{natbib}
\pubyear{2020}
\volume{0}
\issue{0}
\firstpage{1}
\lastpage{8}

\startlocaldefs
\numberwithin{equation}{section}
\theoremstyle{plain}

\endlocaldefs

\begin{document}

\begin{frontmatter}
\title{Disintegration of Gaussian Measures for Sequential Assimilation of Linear Operator Data}
\runtitle{Disintegration of Gaussian Measures}

\begin{aug}
\author{\fnms{C\'edric} \snm{Travelletti}\thanksref{t1}\ead[label=e1]{cedric.travelletti@stat.unibe.ch}}
\and
\author{\fnms{David} \snm{Ginsbourger}\thanksref{t1}\ead[label=e2]{david.ginsbourger@stat.unibe.ch}}

\address{
  Institute of Mathematical Statistics and
    Actuarial Science\\
  University of Bern\\
  Bern, Switzerland\\
\printead{e1,e2}}

\thankstext{t1}{The authors gratefully acknowledge funding from the Swiss National Science Fundation (SNF) through project no. 178858.}
\runauthor{C. Travelletti et al.}

\end{aug}

\begin{abstract}

Gaussian processes appear as building blocks in various stochastic models 
and have been found instrumental to account for imprecisely known, latent functions. 
It is often the case that such functions may be directly or indirectly evaluated, 
be it in static or in sequential settings. 
Here we focus on situations where, rather than pointwise evaluations, 
evaluations of prescribed linear operators at the function of interest 
are (sequentially) assimilated.  
While working with operator data is increasingly encountered in the practice 
of Gaussian process modelling, mathematical details of conditioning and model 
updating in such settings are typically by-passed. 
Here we address these questions by highlighting conditions under which 
Gaussian process modelling coincides with endowing separable Banach spaces of 
functions with Gaussian measures, and by leveraging existing results on the 
disintegration of such measures with respect to operator data. 
Using recent results on path properties of GPs and their connection to RKHS, 
we extend the Gaussian process - Gaussian measure correspondence beyond 
the standard setting of Gaussian random elements in the Banach space of continuous 
functions. Turning then to the sequential settings, 
we revisit update formulae in the Gaussian measure framework and establish 
equalities between final and intermediate posterior mean functions and 
covariance operators.
The latter equalities appear as infinite-dimensional and discretization-independent 
analogues of Gaussian vector update formulae.          	
\end{abstract}


\tableofcontents
\end{frontmatter}

\section{Introduction}\label{sec:introduction}
Gaussian process (GP) stochastic models have found broad use in a variety 
of domains such as filtering, geostatistics, 
or analysis of computer codes. They are also frequently used 
in machine learning as priors on functions for tasks where one tries 
to learn a unknown function $f$ in a Bayesian way. Such machine learning uses 
include, among others, 
Bayesian inversion \citep{jackson_inverse} and Bayesian optimization 
\citep{kushner,mockus_bo,jones}.

One reason for the success of GPs in these domains is their closure under pointwise 
observations. Indeed, given a GP $\gp=(\gp_s)_{s\in \domain}$ on some domain $\domain$ and a 
set of points $s_1, ..., s_n \in \domain$, the distribution of $\gp$ conditionally 
on $\gp_{s_1}, ..., \gp_{s_n}$ is again Gaussian, 
with mean and covariance functions that can be computed in closed form, see e.g. \citep{rasmussen_williams}.

While traditional machine learning tended to focus only on data in the form 
of pointwise evaluations, other types of indirect, functional data become increasingly available, such as tomographic data, or derivative data \citep{solak2003derivative,ribau2018} that do not boil down to simple pointwise evaluations of the original latent function. This has sparked interest in extending GPs to different types of observations, such as integral 
observations \citep{hendriks2018evaluating,jidling2019deep} or 
linear constraints \citep{jidling2017linearly,agrell_jmlr_2019}. Broadly speaking, these methods 
aim at learning $f$ from linear form data $\ell_i(f)$, where $\ell_i:X\rightarrow \mathbb{R}$ ($i=1,2,\dots$) are linear functionals on some Banach space $X$ of functions on $D$. Just like under pointwise observations, working out conditional distributions boils down to applying conditioning formulae to finite-dimensional vectors, in that case to vectors of the form $(Z_{s}, Z_{s'}, \ell_1(Z),\dots,\ell_n(Z))$ ($s,s'\in D$).

Compared to the basic case of pointwise observations, however, ensuring that the usual way of deriving conditional distributions does actually work under linear form data requires a bit of care. The usual approach in practice is to silently assume that the considered functionals of $\gp$ can be expressed as limits of linear combinations of pointwise field evaluations, so that everything will work as intended. 
In several cases, this condition might not be straightforward to verify, and things can get even worse when one considers observations described by linear operators between Banach spaces $G:X \rightarrow Y$, thus raising the question of what kind of operator data can be assimilated, or more precisely, of 
which properties an operator $G$ needs to satisfy in order for 
the conditional law to be well-defined. While this question can be tricky to 
answer using the traditional Gaussian process framework, 
modern probability theory in Banach spaces offers a rigorous, generic approach to conditioning under linear operator using the language of disintegrations 
of measures, as we will clarify next.

Beyond establishing solid mathematical foundations for conditioning on linear operator data, another problem that has received much attention lately in the GP litterature is that of efficiently performing sequential 
data assimilation \citep{opt_des_inverse,huber2014recursive,solin_2015}. 
In such a framework, new data become available sequentially and predictions have to be recomputed along the way to incorporate the new information. 
To alleviate the computational burden associated to sequential learning, 
various \textit{updating} scheme have been developed 
\citep{update_chevalier,update_emery,update_gao,update_barnes} which aim 
at expressing the contribution of the new data as an update to the current 
posterior.

In the present work, we focus on the intersection of the two aforementioned topics, that is, we concentrate on sequential assimilation of linear operator data. Our aim is to provide an abstract mathematical foundation for the above setting by formulating it in the language of disintegrations and to derive update formulae for disintegrations. 
In passing, we clarify the link between the traditional Gaussian process framework and the Gaussian measure language. This work emerged as a theoretical foundation for practical approaches 
to large-scale assimilation of linear operator data under GP priors developed in \citep{travelletti2021}.

The article is structured as follows: in \Cref{sec:process_measure_equivalence}, 
we review results from \citet{rajput_gp_vs_measures} in order to 
prove equivalence of the Gaussian 
process and Gaussian measure approaches in various cases. 
We also connect this 
with recent results on sample path properties of GP \citep{steinwart_mercer_supplementary} 
to characterize situations 
under which GPs induce a Gaussian measure on some suitable space of functions.

Then, in \Cref{sec:disintegration}, we turn to disintegrations of Gaussian measures 
\citep{TARIELADZE2007851}, which we extend to the non-centred and sequential case, thereby 
providing an extension of the usual kriging update formulae \citep{update_chevalier} 
to disintegrations. 

Those results offer prospects for theoretical inquiries in Bayesian optimization 
\citep{bect_2019} as well more applied uses, such as the formulation of discretization-independent 
algorithms in Bayesian inversion \citep{stuart_cotter}. We also hope that our characterization 
of the Gaussian process - Gaussian measure equivalence will help bring benefits of the abstract language 
of disintegrations to the applied GP community.


\begin{example*}
For the rest of this work, we will consider the task 
of learning an unknown function $f$ living in a separable Banach space $X$ 
from data of the form $y_i = G_i(f),~i=1,\dotsc, n$, where 
\begin{align*}
    G_i: X \rightarrow Y,
\end{align*}
are bounded liner operators into a separable Banach space $Y$, we will call 
the $G_i$ the \textit{observation operators}. As a simple example of a problem falling into this 
setting, consider the task of learning a continuous function defined on the interval $\left[-1, 1\right]$ via different types of data: pointwise function values, integrals of the function, Fourier coefficients, etc. \Cref{fig:fourier_variance_reduction} provides an illustration of solutions obtained under a Gaussian process prior. Note that the three different combinations of observations in \Cref{fig:fourier_variance_reduction} can each be described by a linear operator $G:C\left(\left[-1, 1\right]\right)\rightarrow\mathbb{R}^{\datdim}$
\begin{figure}[h!]
\centering
\subfloat[prior]{\includegraphics[scale=0.38]{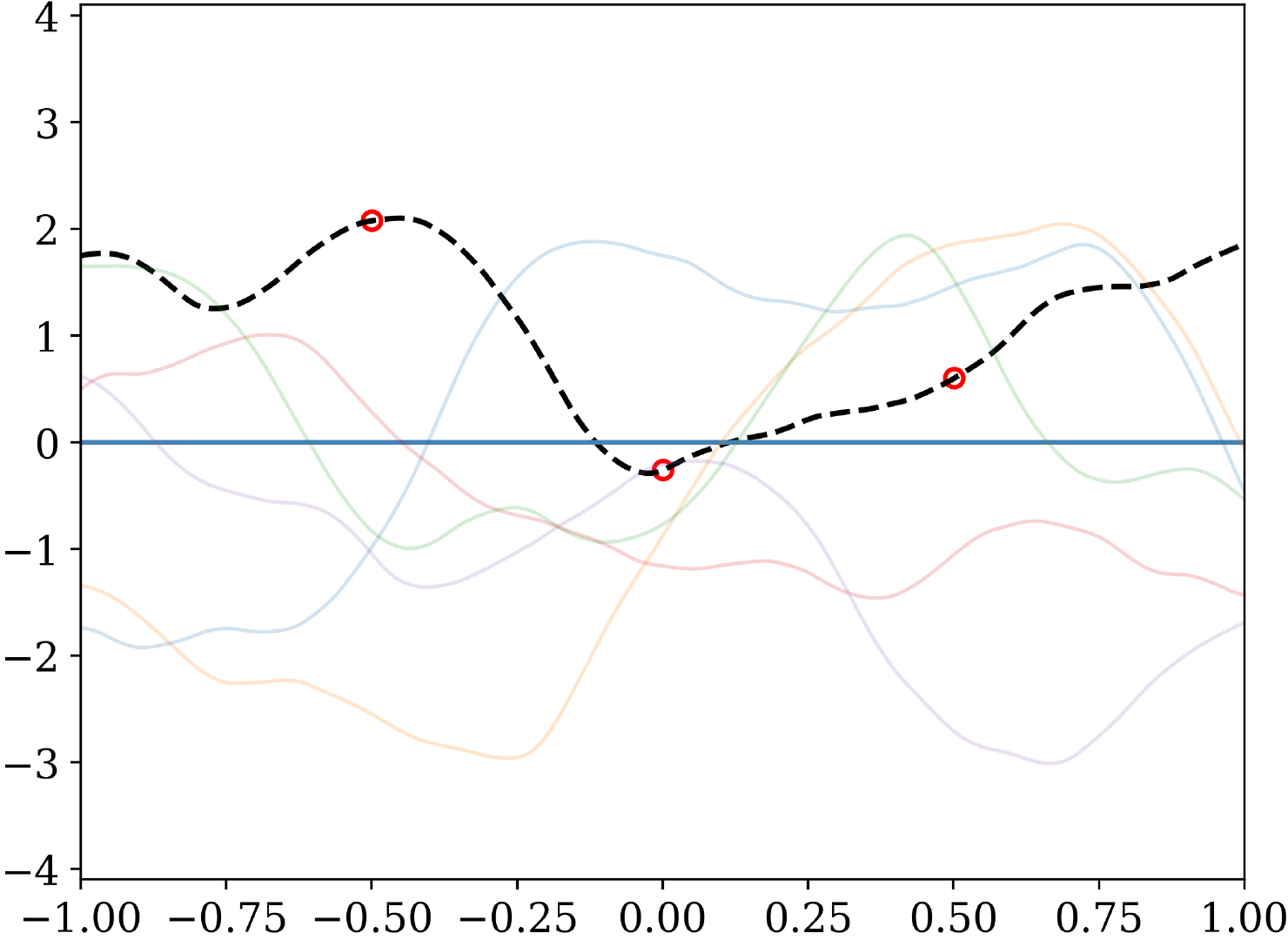}}
\subfloat[pointwise]{\includegraphics[scale=0.38]{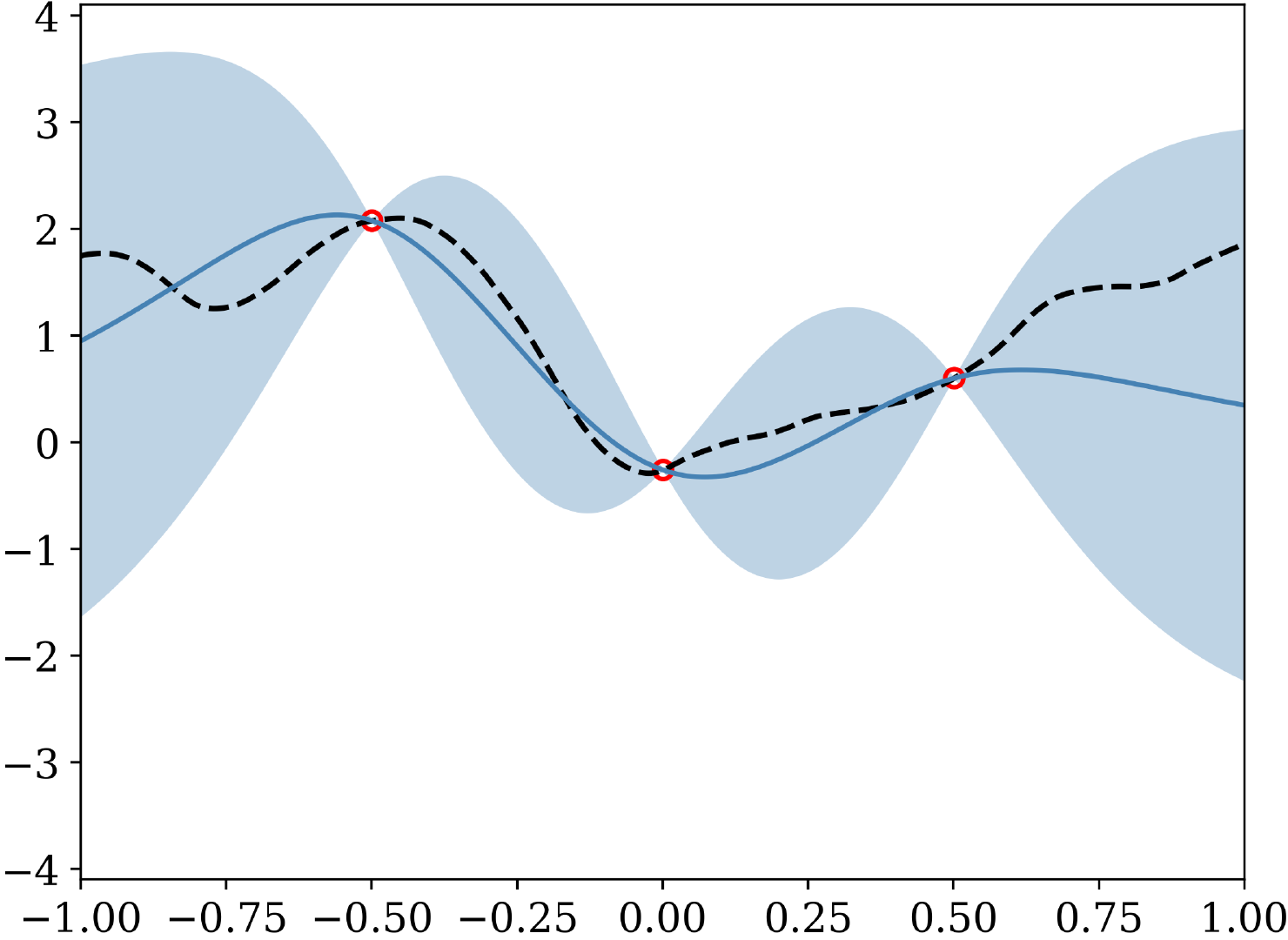}}\\
\subfloat[pointwise + integral]{\includegraphics[scale=0.38]{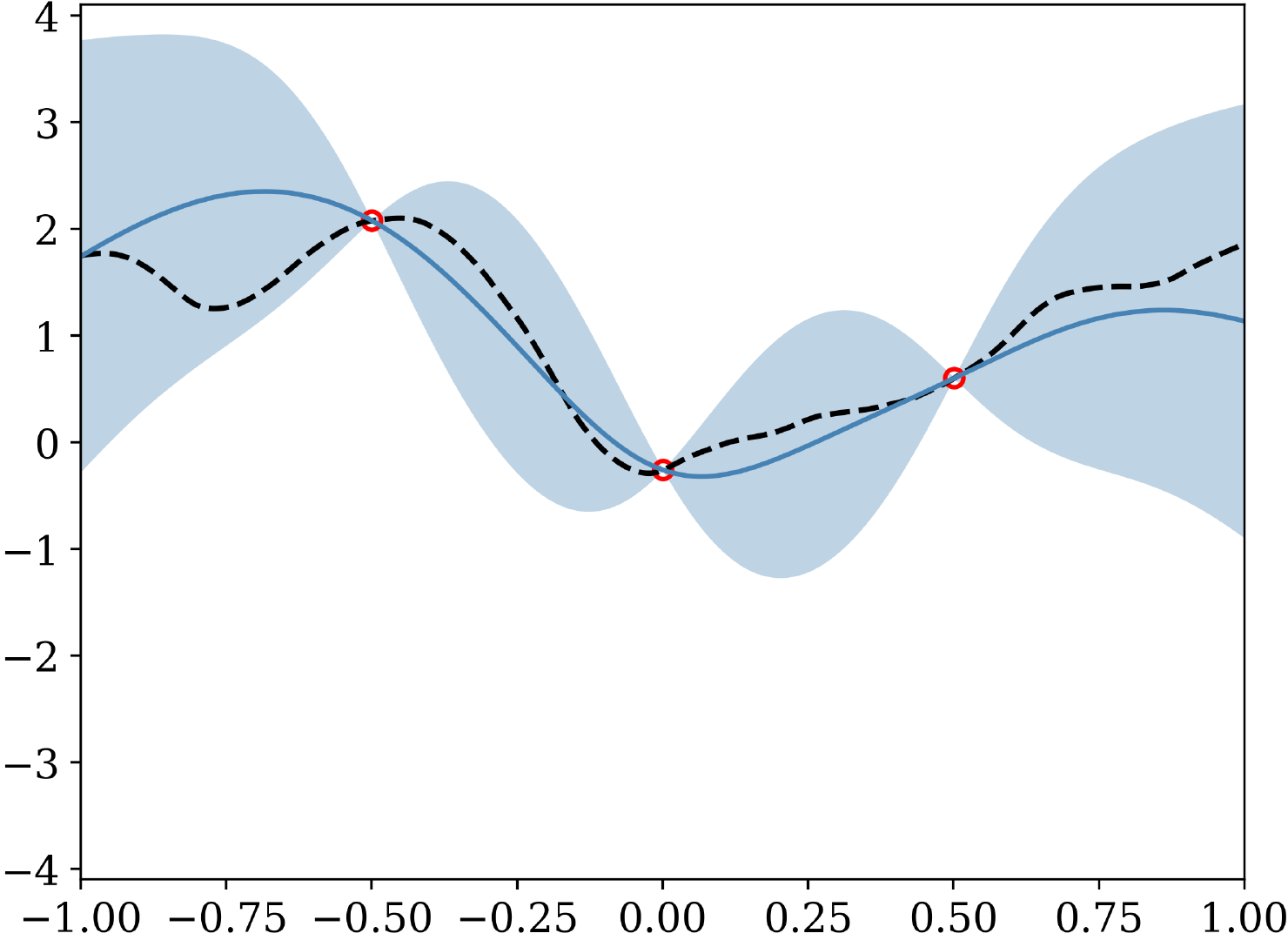}}
\subfloat[pointwise + integral + Fourier]{\includegraphics[scale=0.38]{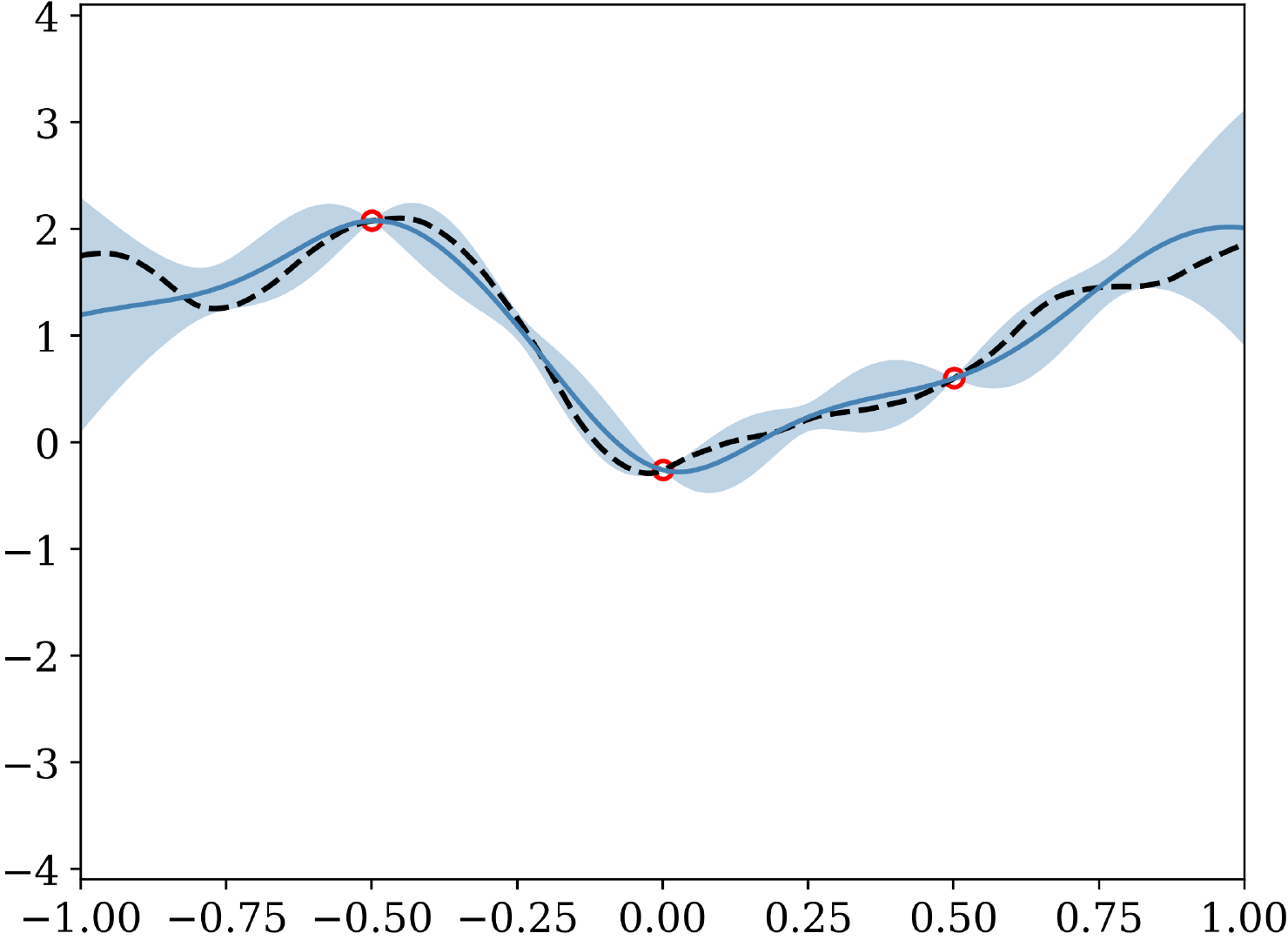}}
\caption{Conditional mean (blue) and 2$\sigma$ credible intervals after inclusion of different types of data: (a) realizations independently sampled from the prior GP, (b) prediction based on pointwise data at 3 locations, (c) prediction based on pointwise data + integral over domain, (d) prediction based on pointwise data + integral over domain + first two Fourier coefficients. The true unknown function is shown in dashed black.
}
\label{fig:fourier_variance_reduction}
\end{figure}
\end{example*}

Note that this example can already serve to illustrate the theoretical difficulties 
associated with the conditional law under linear operator observations. Consider for 
example derivative observations $y = f'(x_0),~x_0 \in D$. The usual procedure when working 
with derivatives of GPs is to assume mean square differentiability of the process. 
But even then, results on the link between mean 
square differentiability of the process and almost sure differentiability of the paths 
\citep{cambanis_differentiability,SCHEUERER20101879} require additional assumptions 
to ensure path differentiability, so that the observation operator is not 
guaranteed to be bounded.


\section{Gaussian Processes and Gaussian Measure: Background and Equivalence}\label{sec:process_measure_equivalence}
When working with Gaussian priors over spaces of function defined over an 
arbitrary domain $D$, two complementary
approaches are often used:

\begin{itemize}
    \item One can work with a  Gaussian process on $D$, which is defined as a
        stochastic process  $Z=(Z_s)_{s \in D}$ indexed by $D$, such that for 
        any number of points  $s_1, ..., s_n \in D$, the distribution of
         $\left(Z_{s_1}, ..., Z_{s_n}\right)$  is Gaussian, e.g., \citep{tarantola,Sarkka_gp_operator}.
    \item One can work with a
         Gaussian measure  which is defined as a Borel measure on  $C(D)$ 
        such that for any continuous linear functional  $\ell \in C(D)^*$,
        the measure  $\ell_{\#}\mu=\mu \circ \ell^{-1}$ on $\mathbb{R}$  is Gaussian, e.g., \citep{stuart_2010,stuart_dashti,Sullivan2015,Ernst2014}.
\end{itemize}
~\\
For a Gaussian process $Z$ as introduced above, its mean and covariance function are defined as
\begin{align*}
      &m: s \in D \mapsto m(s) = \mathbb{E}\left[Z_x\right]\\
      &k: (s,t)\in D^2 \mapsto  k(s, t) = \mathbb{E}\left[Z_s Z_t\right] - \mathbb{E}\left[Z_s\right]\mathbb{E}\left[Z_t\right],
\end{align*}
where the existence of moments is guaranteed by the joint Gaussianity of $(Z_s,Z_t)$ for any $(s,t)\in D^2$. Note that here we will often used the alternative notation $m_s$ for $m(s)$ as it will increase the readability of forthcoming equations. 

When working with a Gaussian measure $\mu$ over a separable Banach space $X$, the notions of mean and covariance functions are respectively replaced by the \textit{mean element} and \textit{covariance operator}. Here we denote by $X^*$ the (continuous) dual space of $X$, and for any element $f\in X$ and continuous linear form $g^*$ we use the duality notation $\langle f,g^* \rangle=g^*(f)$. 

\begin{definition}
Given a Gaussian measure $\mu$ on a Banach space $X$, the mean of $\mu$ is the unique element $m_{\mu}\in X$ such that:
\begin{equation}\label{eq:defmeanelement}
    \int_X \langle f, g^*\rangle d\mu(f) = \langle m_{\mu}, g^*\rangle,~\forall g^*\in X^*.
\end{equation}
The covariance operator of $\mu$ is the linear operator $C_{\mu}:X^*\rightarrow X$ defined by
\begin{equation}\label{eq:defcovarianceoperator}
    \langle C_{\mu}g_1^*, g_2^*\rangle = \int_X 
    \left(
        \langle f, g_1^* \rangle
        - \langle m_{\mu}, g_1^*\rangle
    \right)
        \left(
        \langle f, g_2^* \rangle
        - \langle m_{\mu}, g_2^*\rangle
        \right)
            d\mu(f)
,~\forall g_1^*, g_2^*\in X^*
\end{equation}
We refer the reader to \citet{Vakhania1987} for more details.
\end{definition}


When considering Gaussian processes with continuous trajectories over a compact 
metric space $D$, the Gaussian process and Gaussian measure points of view are known to be equivalent, with $X$ being the Banach space of continuous functions $C(D)$ equipped with the sup norm.
Indeed, one can show that a Gaussian measure on $C(D)$ defines an equivalent Gaussian process on D with continuous trajectories, and vice-versa. This allows one to work with Gaussian measures and Gaussian processes interchangeably on this Banach space. 
The equivalence is ensured by the following two theorems, which are multidimensional analogues of the one presented in \citet{rajput_gp_vs_measures}.

We first show that a Gaussian process on $D$ with continuous sample paths induces a Gaussian measure on $C(D)$. Indeed, given such a Gaussian process $Z$, one may try to induce a measure $ \mu_Z := \mathbb{P}\circ \Phi^{-1}$, where $\Phi:=Z\left(\cdot;\omega\right)$. The next theorem guarantees that this indeed defines a Gaussian measure. This result is well known 
in the Gaussian measure litterature (see e.g. \cite{bogachev1998gaussian}) and we 
provide a proof in the appendix for the sake of completeness.
\begin{theorem}\label{th:gaussianity_induced_measure}
    Let $\left(\Omega,\mathcal{F},\mathbb{P};Z(\omega,s),s\in D\right)$ be a
    Gaussian process on a compact metric space $D$ with continuous sample paths.
    Then the induced measure 
    \[
        \mu_Z := \mathbb{P}\circ \Phi^{-1}
    \]
    is well-defined (as a Borel measure) and Gaussian.
\end{theorem}
On the other hand, given a Gaussian measure $\mu$ on $C(D)$, the following theorem ensures 
that $\mu$ induces indeed a Gaussian process.
\begin{theorem}\label{th:gaussianity_induced_process}
    Let $\mu$ be a Gaussian measure on $C(D)$, for a compact metric space $D$. Then, letting $\Omega=C(D)$ and $\mathcal{F}$ be the Borel sigma algebra on $C(D)$,
    the collection of random variables
    \[
        Z_s:\left(\Omega,\mathcal{F}, \mu\right)\rightarrow
            \left(\mathbb{R}, \mathcal{B}\left(\mathbb{R}\right)\right),~ \omega\mapsto
            \delta_s\left(\omega\right)
    \]
    for all $s \in D$ defines a Gaussian process with paths
    in $C(D)$ which induces $\mu$ on $C(D)$.
\end{theorem}

Under this correspondence, the mean and covariance functions of the process may be obtained as special cases of the mean element and covariance operator of the corresponding measure by acting on them with Dirac delta functionals (which in this case belong to the continuous dual 
of the Banach space under consideration):

\begin{lemma}\label{th:corresp_mean_cov}
    Let $Z$ be a Gaussian process on a compact metric space $D$ with continuous trajectories, and let $\mu$ be the corresponding induced measure on $C(D)$. Then the covariance operator and mean element of the measure are related to the mean and covariance function of the process via
    \begin{align}
        m_s &= \mathbb{E}\left[Z_s\right] = \langle m_{\mu}, \delta_s\rangle\label{eq:corresp_mean},\\
        k(s_1, s_2) &= \mathbb{E}\left[Z_{s_1} Z_{s_2}\right] - \mathbb{E}\left[Z_{s_1}\right]\mathbb{E}\left[Z_{s_2}\right] = \langle C_{\mu}\delta_{s_2}, \delta_{s_1} \rangle\label{eq:corresp_cov},
    \end{align}
    for all $s_1, s_2 \in D$.
\end{lemma}
~\\
These considerations allow us to work interchangeably with the two points of views. 
While in many practical circumstances the GP point of view is sufficient, 
Gaussian measures can be leveraged to provide rigorous updating of GPs 
under linear operator observations, as we will show in \Cref{sec:disintegration}.

\begin{remark}
    The correspondence between Gaussian processes and measures is not limited to the Banach space $C(D)$ of continuous functions over a compact metric space. 
    Indeed \citet{rajput_gp_vs_measures} also prove correspondence for $L^p$ spaces and spaces of absolutely continuous functions. 
    However, the proofs are done on a case by case basis. 
\end{remark}



Even if the Banach space $C(D)$ of continuous functions on a compact domain provides 
a basic setting for the Gaussian process - Gaussian measure equivalence, it often 
proves insufficient when one wants to use this correspondence to tackle conditioning under 
linear operator observations. For example, the differential operator $d/dx$ is not even 
a well-defined operator on $C(D)$. 
For such operators, the natural domains to consider are Sobolev spaces. 
This shows that, in the Gaussian measure framework,
when one wants to assimilate observations that are "finer" than simple 
pointwise evaluations, one has to go beyond the Banach space $C(D)$.
This is what we will do in the following 
section by considering reproducing kernel Hilbert spaces.\\

\noindent \textbf{The Reproducing Kernel Hilbert Space Case:}
The proofs of the process-measure equivalence theorems 
\Cref{th:gaussianity_induced_measure,th:gaussianity_induced_process} 
in the Banach space of continuous functions over a compact domain 
rely on having a characterization of the dual space of the Banach space under consideration, and on being able to approximate elements of the dual via pointwise evaluations. 
Indeed, Gaussian measures on a Banach space are characterized by the Gaussianity of their linear functionals, whereas GPs are 
characterized by the Gaussianity of finite collections of field evaluations, making the link between linear functionals and pointwise evaluations a crucial one in the correspondence.

The natural class of spaces where such a link exists is that of reproducing kernel Hilbert spaces (RKHS) \citep{aronszajn,schwartz,berlinet_thomas_agnan,Kanagawa2018GaussianPA}. 
Indeed, one of the defining properties of RKHS is that their (continuous) dual contain the evaluation functionals, so that one can directly adapt the process-measure correspondence theorems. 
Note that the product measurability is still guaranteed by \Cref{th:continuous_product_meas} since RKHS of functions over a compact metric space are contained in the Banach space of 
continuous functions provided that the reproducing kernel is continuous.

\begin{theorem}\label{th:gaussianity_induced_measure_hilbert}
    Let $\left(\Omega,\mathcal{F},\mathbb{P};Z(\omega,s),s\in D\right)$ be a
    Gaussian process with trajectories 
    in a separable RKHS $\mathcal{H}$ of functions over a compact metric space $D$. 
    Then the induced measure 
    \[
        \mu_Z := \mathbb{P}\circ \Phi^{-1}
    \]
    is well-defined (as a Borel measure) and Gaussian.
\end{theorem}

\begin{theorem}\label{th:gaussianity_induced_process_hilbert}
    Let $\mu$ be a Gaussian measure on a separable 
    RKHS $\mathcal{H}$ of functions 
    over a compact metric space $D$. 
    Then, letting $\Omega=\mathcal{H}$ and $\mathcal{F}$ be the Borel sigma algebra on $\mathcal{H}$,
    the collection of random variables
    \[
        Z_s:\left(\Omega,\mathcal{F}, \mu\right)\rightarrow
            \left(\mathbb{R}, \mathcal{B}\left(\mathbb{R}\right)\right),~ \omega\mapsto
            \delta_s\left(\omega\right)
    \]
    for all $s \in D$ is a Gaussian process with paths
    in $\mathcal{H}$ which induces $\mu$ on $\mathcal{H}$.
\end{theorem}

The question of whether GP sample paths lie in an RKHS has been 
widely studied in the litterature 
\citep{steinwart_mercer,steinwart_mercer_supplementary}. 
One of the most well-known results in this domain is a negative one, namely 
that for a GP with continuous covariance kernel and almost-sure 
sample paths, the probability that the trajectories lie within the 
RKHS associated to the kernel of the process is zero \citep{Driscoll1973TheRK,Lukic2001}. 
Recent works have aimed at finding "larger" RKHS that contain the paths of the process.
It turns out that for a broad class of GPs, one can find 
an 'interpolating' RKHS lying between the RKHS of the kernel of the process and $L^2(\nu)$ (for some measure $\nu$) 
that contains the sample paths almost surely \citep[Corollary 5.3]{steinwart_mercer_supplementary}. 

We here only consider kernels that are bounded on the diagonal: 
$k(s, s) < \infty,\text{ all }s\in D$ (as is the case for all the usual kernels). 
Then, \citet[Lemma 5.1, Theorem 5.3]{steinwart_mercer} guarantee 
that the conditions required for the sample paths to be contained in powers 
of the base RKHS hold. Under these conditions, there are results that guarantee 
the existence of an RKHS containing the trajectories of the process 
with probability $1$. The RKHS depends on the eigenvalues of the operator 
\begin{align*}
    T_k(f) := \int_D k(\cdot, s)f(s)d\nu(s),~ f\in L_2(\nu),
\end{align*}
where $\nu$ is any finite Borel measure supported on $D$. 
The embedding RKHS is then constructed as a power $\mathcal{H}_k^{\theta}$
of the RKHS $\mathcal{H}_{k}$ of the kernel \citep[Definition 4.12]{Kanagawa2018GaussianPA}.

\begin{theorem}{[\citet[Theorem 4.12]{Kanagawa2018GaussianPA}, \citet[Theorem 5.2]{steinwart_mercer_supplementary}]}\label{th:sample_path_rkhs}
    Let $Z$ be a Gaussian process over a compact 
    domain $D \subset \mathbb{R}^d$ 
    with covariance kernel $k$. Let also $\left(\lambda_i, \phi_i\right)_{i\in\mathbb{N}}$ 
        be the eigensystem of the operator $T_k$. Then, provided 
    $\sum_{i\in\mathbb{N}} \lambda_i^{1 - \theta} < \infty$,
    there exists a version of $Z$ whose sample paths 
        lie in $\mathcal{H}_k^{\theta}$ with probability $1$.
\end{theorem}

In particular, for GPs 
with Gaussian kernels or Mat\'ern kernels, one can always find an RKHS 
that contains the sample paths of the GP with probability $1$, as 
the following results from \cite{Kanagawa2018GaussianPA} guarantee:

\begin{corollary}[Squared Exponential Random Fields, \citet{Kanagawa2018GaussianPA}]
    If $Z$ is a Gaussian random field with squared exponential kernel $k$ over 
    a compact domain $D \subset \mathbb{R}^d$ with Lipschitz boundary, then 
    for any $0 < \theta < 1$ there exists a version 
    of $Z$ that lies in $\mathcal{H}^{\theta}_k$ with probability $1$.

\end{corollary}
\begin{corollary}[Mat\'ern Random Fields and Sobolev Spaces, \citet{Kanagawa2018GaussianPA}]\label{th:matern_path}
    When $Z$ is a M\'atern Gaussian random field 
    with Mat\'ern kernel $k^{\mathrm{Mat}}_{\alpha,\lambda}$ 
    of order $\alpha$ and lengthscale $\lambda$
    over a domain $D\subset \mathbb{R}^d$ with Lipschitz boundary, then 
    \cite[Corollary 4.15]{Kanagawa2018GaussianPA} guarantees 
    that there exists a version of $Z$ that lies in 
    $\mathcal{H}_{k^{\mathrm{Mat}}_{\alpha', \lambda'}}$ with 
    probability $1$ for all $\alpha', \lambda'>0$ satisfying 
    $\alpha > \alpha' + d/2$, provided that $D$ satisfies an \textit{interior cone 
    condition} (see \cite[Definition 4.14]{Kanagawa2018GaussianPA}).
\end{corollary}

Wrapping everything together, we can formulate a sufficient condition for a Gaussian process to 
induce a Gaussian measure on its space of trajectories:
\begin{corollary}\label{th:lemma_always_induce}
    Let $\left(\Omega,\mathcal{F},\mathbb{P};Z(\omega,x),x\in D\right)$ be a Gaussian process 
    on a compact metric space $D$ with covariance kernel $k$ that is continuous and 
    bounded on the diagonal. 
    Then there exists $0 < \theta \leq 1$ such that $Z$ induces a Gaussian measure 
    on $\mathcal{H}_k^{\theta}$.
\end{corollary}

\begin{remark}
    Note that the construction of the power of a RKHS depends on the choice of the measure $\nu$. This is 
    not a significant handicap since the goal of \Cref{th:lemma_always_induce} is to show that under 
    given conditions on a GP one can always induce a measure from it. Nevertheless, recent results 
    \citep{karvonen2021small} provide constructions of RKHS containing the sample paths that do not depend 
    on a given measure and are "smaller" than constructions involving powers of RKHS. 
    These constructions are mostly useful in providing more fine-grained descriptions of 
    sample path properties for infinitely smooth kernels \citep[Chapter 2]{karvonen2021small}. 
    We refer the interested reader to the aforementioned litterature for more details.
\end{remark}

\begin{remark}
    In practice, when working with derivative-type observations, it is often preferable to have 
    simple conditions on the covariance kernel that enforce the path to live in some 
    Sobolev space that makes the observation operator under consideration a bounded 
    one. Useful results to that end can be found in \citep{SCHEUERER20101879}. 
    In particular, it is shown that continuity on the diagonal of the generalized mixed 
    derivatives of the covariance kernel up to order $k$ ensures that the 
    sample paths lie in the local Sobolev space $W^{k, 2}_{loc}(D)$ of order $k$ almost-surely 
    \citep[Theorem 1]{SCHEUERER20101879}.
\end{remark}

\section{Disintegration of Gaussian Measures under Operator Observations}\label{sec:disintegration}
Now that we have introduced the equivalence of the process and the measure approaches, we consider the posterior in the Gaussian measure formulation of conditioning. In this setting, conditional laws are defined using the language of \textit{disintegrations} of measures.
The treatment presented here will follow that in \citet{TARIELADZE2007851} and extend some of the theorems therein.

In the following, we will let $X$ be a separable Banach space 
of functions over an arbitrary domain $D$ 
such that the measure-processes correspondence introduced in 
\Cref{sec:process_measure_equivalence} holds, 
and use $\mu$ to denote a Gaussian measure on $X$ and $Z$ for the corresponding 
associated Gaussian process on $D$.
Again $G:X\rightarrow \mathbb{R}^{\datdim}$ will denote a bounded linear operator. 

\begin{definition}\label{def:disintegration}
Given measurable spaces $\left(X, \mathcal{A}\right)$ and $\left(Y, \mathcal{C}\right)$, a probability measure $\mu$ on $X$ and a measurable mapping $G:X\rightarrow Y$, a disintegration of $\mu$ with respect to $G$ is a mapping $\tilde{\mu}:\mathcal{A}\times Y \rightarrow \left[0, 1\right]$ satisfying the following properties:
\begin{enumerate}
    \item For each $y \in Y$ the set function $\tilde{\mu}(\cdot, y)$ is a probability measure on $X$ and for each $A \in \mathcal{A}$ the function $\tilde{\mu}(A, \cdot)$ is $\mathcal{C}$-measurable.
    \item There exists $Y_0 \in \mathcal{C}$ with $\mu\circ G^{-1}\left(Y_0\right)=1$ such that for all $y\in Y_0$ we have $\lbrace y \rbrace \in \mathcal{C}$ and for each $y\in Y_0$, the probability measure $\tilde{\mu}(\cdot, y)$ is concentrated on the fiber $G^{-1}\left(\left\lbrace y \right\rbrace\right)$ that is:
    \[
    \tilde{\mu}\left(G^{-1}\left(\left\lbrace y \right\rbrace\right), y\right) = 1.
    \]
    \item The measure $\mu$ may be written as a mixture of the family 
        $\left(\tilde{\mu}(\cdot, y)\right)_{y\in Y}$ with respect to the mixing 
        measure $\mu\circ G^{-1}$:
    \[
    \mu\left(A\right) = \int_Y \tilde{\mu}\left(A, y\right)d\left(\mu\circ G^{-1}\right)\left(y\right),~\forall A \in \mathcal{A}.
    \]
\end{enumerate}
We will use the notation $\mu_{|G=y}\left(\boldsymbol{\cdot}\right):=\tilde{\mu}\left(\boldsymbol{\cdot}, y\right)$ for the \textit{disintegrating measure}.
\end{definition}

The computation of the posterior then amounts to computing a disintegration of the prior with respect to the observation operator.
The existence of the disintegration is guaranteed by Theorem 3.11 in \citet{TARIELADZE2007851}, which we slightly generalize here to non-centered measures.

\begin{theorem}\label{th:tarieladze_disintegration}
    Let $X$, $Y$ be real separable Banach spaces and $\mu$ be a Gaussian measure on the Borel $\sigma$-algebra $\mathcal{B}\left(X\right)$ with mean element $m_{\mu}\in X$ and covariance operator $C_{\mu}:X^*\rightarrow X$. Let also $G:X\rightarrow Y$ be a bounded linear operator. Then, provided that the operator $C_{\nu} := GC_{\mu}G^*:Y^*\rightarrow Y$ has finite rank $\datdim$,
    there exists a continuous affine map $\tilde{m}_{\mu}:Y\rightarrow X$, a symmetric positive operator $\tilde{C}_{\mu}:X^*\rightarrow X$ and a disintegration $\left(\mu_{|G=y}\right)_{y\in Y}$ of $\mu$ with respect to $G$ such that for each $y\in Y$ the measure $\mu_{|G=y}$ is Gaussian with mean element $\tilde{m}_{\mu}(y)$ and covariance operator $\tilde{C}_{\mu}$. Furthermore, 
    for any $C_{\nu}$-representing sequence $y_i^*,~i=1, ..., n$, 
    the mean and covariance are equal to
    \begin{align}
        \tilde{m}_{\mu}(y) &= m_{\mu} + \sum_{i=1}^{\datdim} \left\langle y - Gm_{\mu}, y_i^* \right\rangle C_{\mu} G^* y_i^*\label{eq:post_mean_disintegration}\\
        \tilde{C}_{\mu} &= C_{\mu} - \sum_{i=1}^{\datdim} \left\langle C_{\mu}G^* y_i^*, \boldsymbol{\cdot} \right\rangle C_{\mu} G^* y_i^*\label{eq:post_cov_disintegration}.
    \end{align}
    The mean element also satisfies $G\tilde{m}_{\mu}(y) =y$ for all $y\in Y_0:=Gm_{\mu}+C_{\nu}(Y^*)$.
\end{theorem}
~\\
The explicit formulae for the posterior mean and covariance provided by the above theorem require the use of \textit{representing sequences}.
\begin{definition}{\cite{TARIELADZE2007851}}
Given a Banach space $X$ and a symmetric positive operator $R:X^*\rightarrow X$, a family $(x_i^*)_{i\in I}$ of elements of $X^*$ is called $R$-representing if the following two conditions hold:
\begin{align*}
    \langle R x_i^*, x_j^* \rangle &= \delta_{ij},\\
    \sum_{i\in I} \langle R x_i^*, x^*\rangle^2 &= \langle R x^*, x^*\rangle,~\forall x^* \in X^*.
\end{align*}
\end{definition}
\begin{remark}\label{th:remark_rep_seq}
    In the case where $X$ is a finite-dimensional Hilbert space of dimension $\datdim$,
    one can explicitly compute an $R$-representing sequence by defining 
$x_i^*:=R^{-1/2}e_i,~i=1, ..., \datdim$ where $e_i,~i=1,...,\datdim$ is an orthonormal basis of $X$ (see \Cref{sec:math_and_proofs} for a proof). This fact will be used to link the posterior provided by \Cref{th:tarieladze_disintegration} to the usual formulae for Gaussian processes in the case 
of finite-dimensional data.
\end{remark}

\noindent Using \Cref{th:corresp_mean_cov} we can translate the disintegration provided by \Cref{th:tarieladze_disintegration} to the language of Gaussian processes in the case where $X$ is the Banach space $C(D)$ of continuous functions over a compact metric space $D$:

\begin{corollary}\label{th:disintegration_gp}
Let $Z$ be a Gaussian process on some domain $D$ with 
trajectories in a space $X$ such that either of the equivalence theorems 
\Cref{th:gaussianity_induced_measure} or \Cref{th:gaussianity_induced_measure_hilbert} 
hold. 
Furthermore, let $G:X\rightarrow Y$ be a linear bounded operator into 
a real separable Banach space $Y$. 
Denote by $C_{\mu}$ the covariance operator of the measure associated to the process $Z$. 
Provided the operator $C_{\nu}:= GC_{\mu}G^*$ has finite rank $\datdim$, then, 
for all $y \in Y$ the conditional law of $Z$ given $GZ=y$ is Gaussian with 
mean and covariance function given by, for all $s,s_1,s_2\in D$:
\begin{align*}
    \tilde{m}_s(y) &= \langle \tilde{m}_{\mu}(y), \delta_s \rangle 
    = m_s + \sum_{i=1}^{\datdim} \langle y - G m_{\boldsymbol{\cdot}},
    y^*_i \rangle (C_{\mu}G^*y^*_i)|_s,\\
    \tilde{k}(s_1, s_2) &= \langle \tilde{C}_{\mu} \delta_{s_2}, \delta_{s_1} 
    \rangle = k(s_1, s_2) -
    \sum_{i=1}^{\datdim} (C_{\mu}G^*y^*_i)|_{s_2} ~ (C_{\mu} G^*y^*_i)|_{s_1},
\end{align*}
where $m_s$ denotes the mean function of $Z$ and $Gm_{\boldsymbol{\cdot}}$ 
denotes application of the operator $G$ to the mean function seen as 
an element of $X$ and $\left(y_i^*\right)_{\i=1, ...,\datdim}$ is 
any $C_{\nu}$-representing sequence.
\end{corollary}

\textbf{Link to Finite-Dimensional Case:} 
When $G$ maps into a finite-dimensional Euclidean space and 
$X=C(D)$ for some compact metric space $D$, 
then one can explicitly compute representing sequences and duality pairings, 
allowing the conditional mean and covariance in \Cref{th:disintegration_gp} 
to be entirely written in terms of the prior mean and covariance function 
of the process, making the link to the Gaussian process conditioning formulae 
as found for example in \citet{tarantola}. 
Indeed, since the dual of $C(D)$ is the space of Radon measures on $D$, 
any bounded linear operator $G:C(D)\rightarrow \mathbb{R}^{\datdim}$ 
may be written as a collection of integral operators 
$GZ = \left(\int_D Z(s) d\lambda_i\left(s\right)\right)_{i=1, ..., \datdim}$ 
where the $\lambda_i$'s are Radon measures on $D$. 
This special form allows us to compute closed-from expressions for the 
conditional mean and covariance.

\begin{corollary}\label{th:link_tarantola}
Consider the situation of \Cref{th:disintegration_gp} and let $G:X\rightarrow \mathbb{R}^{\datdim}$.
Then the conditional law of $Z$ given $GZ=y$ is Gaussian with mean and covariance function given by, for all $s, s_1, s_2\in D$:
\begin{align}
    \tilde{m}_s(y) &= m_s
    - K_{sG}
    K_{GG}^{-1}
    \left(y - Gm_{\boldsymbol{\cdot}}\right)\label{eq:cond_mean_field_fin_dim},\\
    \tilde{k}(s_1, s_2) &= k(s_1, s_2) - K_{s_1G} K_{GG}^{-1} K_{s_2G}^T\label{eq:cond_cov_field_fin_dim}
\end{align}
where we have defined the following vectors and matrices:


 \begin{align}
     K_{sG} :&= \left(G_i k\left(\cdot, s\right)\right)_{i=1,..., \datdim}\in\mathbb{R}^{\datdim}\label{eq:forward_pt_cov_matrix},\\
     K_{GG} :&= \left(
         G_i\left(G_j k\left(\cdot, \cdot\right)\right)
     \right)_{i,j=1, ..., \datdim} \in \mathbb{R}^{\datdim \times \datdim},\label{eq:forward_cov_matrix}
 \end{align}

where $k\left(\boldsymbol{\cdot}, \boldsymbol{\cdot}\right)$ denotes the covariance function of $Z$. 
This corollary provides a Gaussian measure-based justification to previously 
used formulae \citep{sarkka2011linear,jidling2018integral,Purisha_2019,longi_2020}.
\end{corollary}
The above corollary provides rigorous formulae for the conditional law under linear operator 
observations when the GP has trajectories that lie either in $C(D)$ or in some RKHS.


\textbf{Sequential Disintegrations and Update:}
We now turn to the situation where several stages of conditioning are performed sequentially. Let again $X$ be a real separable Banach space and consider two bounded linear operators $G_1:X\rightarrow Y_1$ and $G_2: X \rightarrow Y_2$, where $Y_1$ and $Y_2$ are also real separable Banach spaces. Then, if one views these operators as defining two stages of observations, there is two ways in which one can compute the posterior.

\begin{itemize}
    \item On the one hand, one can compute it in two steps by first computing the disintegration of $\mu$ under $G_1$ and then, for each $y_1\in Y_1$, compute the disintegration of $\mu_{|G_1=y_1}$ under $G_2$. 
    \item On the other hand, one can compute it in one go by considering the disintegration of $\mu$ with respect to the \textit{bundled} operator $G:X\rightarrow Y_1 \bigoplus Y_2,~x\mapsto G_1\left(x\right) \bigoplus G_2\left(x\right)$. From now on, we will denote this operator by $\left(G_1, G_2\right)$.
\end{itemize}
~\\
We show that these two approaches yield the same disintegration, as guaranteed by the following theorem.

\begin{theorem}\label{th:transitivity_disintegration}
    Let $X, Y_1,Y_2$ be real separable Banach spaces, $\mu$ be a Gaussian measure on
    $\mathcal{B}\left(X\right)$ with mean element $m_{\mu}$ and covariance
    operator $C_{\mu}:X^*\rightarrow X$. Also let $G_1:X\rightarrow
    Y_1$
    and $G_2:X\rightarrow Y_2$ be bounded linear operators.
    Suppose that both $C_{\nu_1}:= G_1 C_{\mu} G_1^*$ and $C_{\nu_2} :=
    G_2 C_{\mu} G_2^*$ have finite rank $\datdim_1$ and $\datdim_2$, respectively.
    Then
    \[  
        \mu_{|\left(G_1, G_2\right) = \left(y_1, y_2\right)}
        = \left(\mu_{|G_1=y_1}\right)_{|G_2=y_2},
    \]
    where the equality holds for almost all $(y_1, y_2)\in
    Y_1 \bigoplus Y_2$
    with respect to the pushforward measure $\mu\circ \left(G_1,
    G_2\right)^{-1}$ on $Y_1 \bigoplus Y_2$.
\end{theorem}
This theorem can be viewed as a measure-theoretic counterpart to the update formulae for GPs. 
Since both disintegrating measures are equal, it follows that their moments are equal too, 
we can thus characterize sequential disintegration in terms of mean element and covariance 
operator.
Indeed, for the special case of GPs with trajectories in the Banach space of continous functions 
on a compact domain with finite-dimensional data, we can provide explicit update formulate, 
this yields, using \Cref{th:link_tarantola}:

\begin{corollary}\label{th:update_gp}
Let $Z$ be a Gaussian process on a compact metric space $D$ with continuous trajectories. Consider two observation operators $G_1:C(D)\rightarrow \mathbb{R}^{\datdim_1},~\left(G_1 Z\right)_i = \int_D Z_s d\lambda_i^{(1)}$ and $G_2:C(D)\rightarrow \mathbb{R}^{\datdim_2},~\left(G_2 Z\right)_i = \int_D Z_s d\lambda_i^{(2)}$. Denote by $m_{\boldsymbol{\cdot}}$ and $k\left(\boldsymbol{\cdot}, \boldsymbol{\cdot}\right)$ the mean and covariance function of $Z$. Then, for any $y = \left(y_1, y_2\right) \in \mathbb{R}^{\datdim_1 + \datdim_2}$ and any $s, s_1, s_2 \in D$, we have:
\begin{align*}
    \tilde{m}_s(y)
    &=
    m_{s} + 
    K_{sG_1}K_{G_1G_1}^{-1}\left(y_1 - G_1m_{\boldsymbol{\cdot}}\right)
    + K_{sG_2}\left(\tilde{K}^{(1)}_{G_2G_2}\right)^{-1}\left(y_2 - G_2\tilde{m}^{(1)}_{\boldsymbol{\cdot}}\right),\\
    \tilde{k}(s_1, s_2)
    &=
    k(s_1, s_2) 
    - K_{s_1 G_1}K_{G_1G_1}^{-1}K_{s_2 G_1}^T
    - \tilde{K}^{(1)}_{s_1 G_2}\left(\tilde{K}^{(1)}_{G_2G_2}\right)^{-1}\left(\tilde{K}^{(1)}_{s_2 G_2}\right)^T,
\end{align*}
where $G:=\left(G_1, G_2\right)$ and $\tilde{m}^{(1)}$ denotes the conditional mean 
of $Z$ given $G_1Z=y_1$ as given by \Cref{th:link_tarantola}. Also 
$\tilde{K}^{(1)}_{G_2G_2}$ and $\tilde{K}^{(1)}_{sG_2}$ denote the same matrices as in \Cref{eq:forward_pt_cov_matrix,eq:forward_cov_matrix} with the prior covariance $k\left(\boldsymbol{\cdot}, \boldsymbol{\cdot}\right)$ replaced by the conditional covariance of $Z$ given $G_1Z$.
\end{corollary}

\textbf{Infinite Rank Data:}
For the sake of completeness, we also consider sequential conditioning 
in the presence of 'infinite rank data'. That is, we want to adapt 
\Cref{th:tarieladze_disintegration} and its corrolaries, as well as 
\Cref{th:transitivity_disintegration} to the case where 
$C_{\nu} := GC_{\mu}G^*:Y^*\rightarrow Y$ does not have finite rank. 
Thanks to \cite[Lemma 3.5]{TARIELADZE2007851} we are still able to find 
a $C_{\nu}$-representing sequence and \cite[Lemma 3.4]{TARIELADZE2007851} 
guarantees the convergence of the series defining the covariance operator. 
The main difference compared to the finite rank case is that we can only 
define the disintegration on a full measure subspace of the data:

\begin{theorem}\label{th:tarieladze_disintegration_infinite}
    Let $X$, $Y$, $\mu$, $G$, $\nu$ and $C_{\nu}$ 
    be as in \Cref{th:tarieladze_disintegration} and 
    assume that $C_{\nu}$ has infinite rank. 
    Then there exists a subspace $Y_0$ of $Y$ with $\nu(Y_0)=1$ and 
    a disintegration $\left(\mu_{|G=y}\right)_{y\in Y_0}$ of $\mu$ with 
    respect to $G$ such that for each $y \in Y_0$ the measure $\mu_{|G=y}$ is 
    Gaussian with mean element and covariance operator:
    \begin{align}
        \tilde{m}_{\mu}(y) &= m_{\mu} + \sum_{i=1}^{\infty} \left\langle y - Gm_{\mu}, y_i^* \right\rangle C_{\mu} G^* y_i^*\label{eq:post_mean_disintegration_infinite}\\
        \tilde{C}_{\mu} &= C_{\mu} - \sum_{i=1}^{\infty} \left\langle C_{\mu}G^* y_i^*, \boldsymbol{\cdot} \right\rangle C_{\mu} G^* y_i^*\label{eq:post_cov_disintegration_infinite},
    \end{align}
    where $\left(y_i^*\right)_{i\in\mathbb{N}}$ is any $C_{\nu}$-representing sequence.
    Furthermore, the map $\tilde{m}_{\mu}:Y_0 \rightarrow X$ is continuous and affine 
    and the mean element satisfies $G\tilde{m}_{\mu}(y) =y$ for all $y\in Y_0:=Gm_{\mu}+C_{\nu}(Y^*)$.
\end{theorem}

Concerning the transitivity of disintegrations in the infinite rank 
data setting, one sees that \Cref{th:transitivity_disintegration} 
holds with only slight modifications. 
Indeed, the only necessary 
adaptation is that one should restrict the joint disintegration to the 
direct sum of the subspaces where the individual disintegrations are defined, but 
since those are of full measure, the conclusion of the theorem still holds.

\begin{theorem}\label{th:seq_cond_inf}
     Let $X, Y_1,Y_2$ be real separable Banach spaces, $\mu$ be a Gaussian measure on
     $\mathcal{B}\left(X\right)$ with mean element $m_{\mu}$ and covariance
     operator $C_{\mu}:X^*\rightarrow X$. Also let $G_1:X\rightarrow
     Y_1$
     and $G_2:X\rightarrow Y_2$ be bounded linear operators. 
     Then there exists a subspace $Y_0:=(Y_0^{(1)}, Y_0^{(2)})\subset Y$ such 
     that $\nu(Y_0)=1$ and for all $(y_1, y_2) \in (Y_1^{(0)}, Y_2^{(0)})$ we have:
     \[  
         \mu_{|\left(G_1, G_2\right) = \left(y_1, y_2\right)}
         = \left(\mu_{|G_1=y_1}\right)_{|G_2=y_2}.
     \]
 \end{theorem}
 This theorem provides a rigorous basis for Gaussian process update 
 in the case of infinite rank data. We stress that assimilation of such data can be theoretically challenging when using the standard Gaussian process framework, which relies on linear combinations of pointwise field evaluations to define conditional laws. We believe the above showcases the convenience of the measure-disintegration framework and how it can handle such type of data more naturally.
 We hope this can serve as a basis for further contributions.\\

 As a final byproduct, one can write update formulae for sequential conditioning 
 (disintegration) of Gaussian measures in terms of their moments. Denoting by 
 $m_{\mu}^{(1)}(y_1)$ and $C_{\mu}^{(1)}$ the mean element and covariance operator 
 of the disintegrating measure $\mu_{|G_1=y_1}$ and by 
 $m_{\mu}^{(1\oplus 2)}(y_1, y_2)$, 
 respectively $C_{\mu}^{(1\oplus 2)}$ those of the disintegration measure 
 $\mu_{|(G_1, G_2)=(y_1, y_2)}$ one obtains the following corollary.

 \begin{corollary}\label{th:ultimate_update}
     Consider the same setting as \Cref{th:seq_cond_inf} and let 
     $(y_i^{*(2)})_{i=1,...,p_{2}}$ be any
    $G_2 C^{(1)}_{\mu} G_2^*$-representing sequence.
     Then the mean element and covariance operator of the 
     disintegrating measure $\mu_{|(G_1, G_2)=(y_1, y_2)}$ can be written 
     in terms of the moments of the intermediate disintegrating measure 
     $\mu_{|G_1=y_1}$ as:
     \begin{align*}
         m_{\mu}^{(1\oplus 2)}\left(y_1, y_2\right) &= 
         m_{\mu}^{(1)}\left(y_1\right)
         +
         \sum_{i=1}^{\infty}
         \left\langle
             y_2 - G_2 m_{\mu}^{(1)}\left(y_1\right), y_i^{(2)*}
        \right\rangle
        C_{\mu}^{(1)}G_2^*y_i^{(2)*}\\
         C_{\mu}^{(1\oplus 2)} &= C_{\mu}^{(1)} - \sum_{i=1}^{\infty}
         \left\langle
             C_{\mu}^{(1)}G_2^*y_i^{(2)*}, \cdot
        \right\rangle
        C_{\mu}^{(1)}G_2^*y_i^{(2)*},
     \end{align*}
     where the equalities hold for almost all $(y_1, y_2) \in Y_1 \bigoplus Y_2$ 
     with respect to $\mu \circ (G_1, G_2)^{-1}$.
 \end{corollary}

 Note that this corollary provides an extension to Gaussian measures and operator observations 
 of the well-known kriging update formulae \citep{update_chevalier} and can be viewed as 
 subsuming various gaussian conditioning update formulae under a rigorous and abstract theoretical framework.

\begin{example*}[continued]
    We now come back to the example from the introduction to demonstrate the machinery developed in the two preceding sections.
    Assume that we want to add derivative observation at $x=0$.\\
    First, in order to apply the disintegration theorems, we need to make sure 
    that the observation operator under consideration is a bounded operator on a Banach space in which the path of the prior 
    lie with probability one. In this example, the prior that was used was a Mat\'ern $5/2$ GP with 
    lengthscale parameter $\lambda=0.4$. According to \Cref{th:matern_path}, the path of the prior almost surely lie in the 
    Sobolev space $H_2([-1, 1])$, so taking $X=H_2([-1, 1])$ ensures that the observation operators are bounded (integral 
    and Fourier observations are bounded since the domain is compact and the paths continuous).

\begin{figure}[h!]
\centering
\subfloat[pointwise + integral + Fourier + derivative]{\includegraphics[scale=0.58]{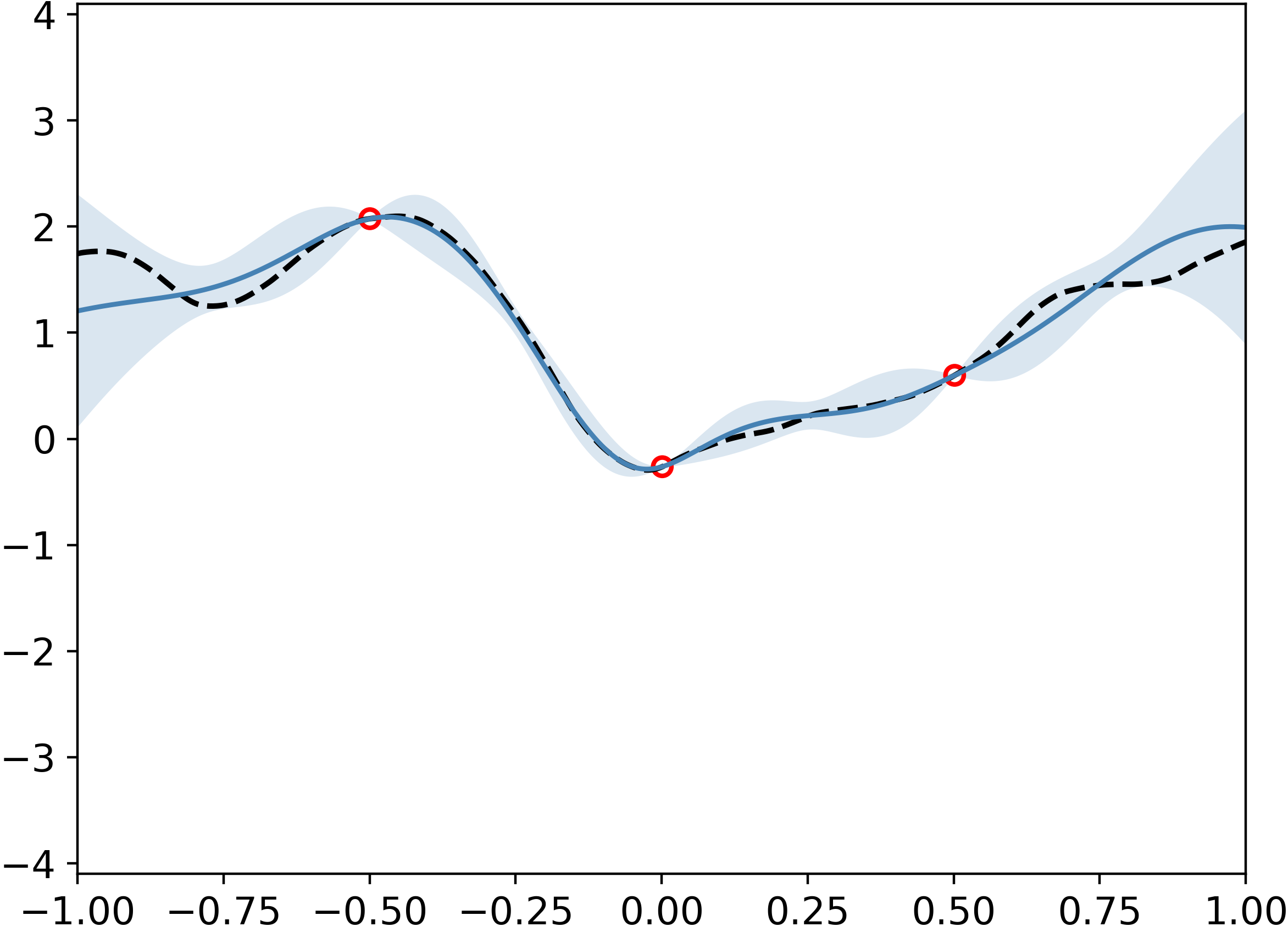}}
\caption{Continuation of the introductory example with addition of derivative observation at $x=0$.
}
\label{fig:fourier_variance_reduction_deriv}
\end{figure}
\end{example*}

\noindent Now, $H_2([-1, 1])$ is a RKHS and thus by \Cref{th:gaussianity_induced_measure_hilbert} the Gaussian measure - Gaussian process correspondence 
is applicable. Furthermore, the 7 observations (3 pointwise + 1 integral + 2 Fourier + 1 derivative) considered can be described 
by a bounded operator between separable Banach spaces $G: H_2([-1, 1]) \rightarrow \mathbb{R}^7$, so that the 
disintegration framework from \Cref{sec:disintegration} can be used. 
Finally, using the updated formulae (\Cref{th:ultimate_update}) one can express 
the posterior mean and covariance after inclusion of the derivative observation as an update of the one after assimilation of the previous observations:

\begin{align*}
    \tilde{m}^{(7)}_{x_0}(y_7) &= \tilde{m}_{x_0}^{(6)}(y_1, ...,y_6)\\
                               &\phantom{--}+
        \frac{d}{dx} \tilde{k}^{(6)}(x_0, x)\vert_{x=0}
        \left(
            \frac{d}{dx}\frac{d}{dx'}\tilde{k}^{(6)}(x, x')\vert_{x,x'=0}
        \right)^{-1}
        (y_7 - \frac{d}{dx}\tilde{m}_x^{(6)}(y_1, ...,y_6)\vert_{x=0})
        \\
        \tilde{k}^{(7)}(x_1, x_2) &=
        \tilde{k}^{(6)}(x_1, x_2)
        -
        \frac{d}{dx} \tilde{k}^{(6)}(x_1, x)\vert_{x=0}
        \left(
            \frac{d}{dx}\frac{d}{dx'}\tilde{k}^{(6)}(x, x')\vert_{x,x'=0}
        \right)^{-1}
        \frac{d}{dx} \tilde{k}^{(6)}(x, x_2)\vert_{x=0}
\end{align*}
where $\tilde{m}_{x_0}^{(6)}(y_1, ...,y_6)$ and $\tilde{k}^{(6)}(x_1, x_2)$ 
denote the mean and covariance function after inclusion of the first 6 observations. 
Note that the correspondence between the mean element and covariance operator of the induced 
measure and the mean and covariance function of the process (\Cref{th:corresp_mean_cov}) 
can be used since the Dirac delta functionals belong to the dual of $H_2([-1, 1])$. 
This example demonstrates how the Gaussian measure framework can be used to provide 
a thorough theoretical grounding to previously known techniques 
\citep{solak2003derivative,ribau2018,agrell_jmlr_2019}.

\section{Conclusion} 
By bridging recent results about GP sample path properties with the framework of Gaussian measures, we provide a formulation of sequential data assimilation of linear operator data under Gaussian models in the language of disintegrations of measures. We show equivalence of the Gaussian process and Gaussian measure approaches 
and generalize the GP update formulae to disintegrations. While providing a purely functional formulation of the assimilation process, the framework of disintegrations also allows for a more rigorous abstract treatment of the conditional law. 
This can be leveraged to provide fast update formulae for GP under linear 
operator observations \citep{travelletti2021} 
and we hope 
it can serve as foundations for further theoretical enquiries and practical developments in probabilistic function modelling.


\appendix

\section{Proofs of Equivalence of Gaussian Process and Gaussian measure}\label{sec:equiv_process_measure}

We here briefly recall the theorems and definitions needed 
to prove our main results, and present the proofs. 
For the functional analysis background, we refer the reader to \citet{folland} and to \citet{TARIELADZE2007851,Vakhania1987} for the background about Gaussian measures. The theorems for equivalence between Gaussian processes and Gaussian measures are adapted from \citet{rajput_gp_vs_measures}, while the one for conditioning / disintegration of Gaussian measures are adapted from \citet{TARIELADZE2007851}.

Most of this chapter will be concerned with random variables taking values in the space of continuous function $C(D)$, where $D$ is a compact metric space. When endowed with the sup-norm, $C(D)$ turns into a Banach space. This space enjoys two useful properties:
\begin{enumerate}
    \item $C(D)$ is separable, and as a consequence, the Borel $\sigma$-algebra and the cylindrical $\sigma$-algebra on
        $C(D)$ agree.
    \item The dual space $C(D)^*$ is the space of Radon measures on $D$ and (by Riesz-Markov-Kakutani \citep{rudin}) for
        all $\ell \in C(D)^*:\exists \lambda$ Radon measure on $D$ such that
        \[
            \forall f\in C(D): \ell(f) = \int f d\lambda.
        \]
\end{enumerate}

In order to prove \Cref{th:gaussianity_induced_measure} 
and \Cref{th:gaussianity_induced_process}, we first recall a classic approximation 
result for continuous real-valued functions on compact metric spaces that will 
be useful for proving measurability properties and Gaussianity of the measure induced by a GP. 
For reference, see \citep[Theorem 2.10]{folland}.
\begin{lemma}\label{th:proof_approx_box}
    Let $D$ be a compact metric space and $f:D\rightarrow \mathbb{R}$ be continuous. 
    Then, there exists a sequence of simple functions $f_n$ converging to $f$ uniformly 
    on $D$. For each $n$, the approximating function can be written as:
    \begin{align}
        f_n &= \sum_{k=0}^{K(n)}
        f\left(t_k^{(n)}\right)\mathbb{1}_{A^{(n)}_{k}},\label{eq:proof_approx_box}
    \end{align}
        where $K(n)\in\mathbb{N}$, $t_k^{(n)}\in D$ and the $A_k^{(n)}$'s are 
        Borel measurable sets for all $k$.
\end{lemma}
~\\
We now show that, for stochastic processes on compact metric spaces, having continuous
sample paths is enough to ensure product measurability.
\begin{theorem}\label{th:continuous_product_meas}
    Let $\left(\Omega,\mathcal{F},\mathbb{P};Z(x;\omega),x\in D\right)$ be a
    stochastic process on a compact metric space $D$ with continuous sample paths.
    Then it is measurable as a mapping $\left(D\times \Omega,~
        \mathcal{B}(D)\times \mathcal{F}\right)\rightarrow
        \left(\mathbb{R},~\mathcal{B}(\mathbb{R})\right)$ (product measurable).
\end{theorem}
\begin{proof}
    This is a direct consequence of \citet[Theorem 2]{gowrisankaran_measurability_product}.
\end{proof}

We now have all the ingredients to prove the main theorems about equivalence of process and measure.
\begin{proof}{(\Cref{th:gaussianity_induced_measure})}
    By \Cref{th:continuous_product_meas}, the only thing left to prove is that
    for all $\ell \in C(D)^*$ the real random variable $\ell \circ \Phi$ is Gaussian.\\
    By the Riesz-Markov representation theorem, there exists a Radon measure
    $\lambda$ on $D$ representing $\ell$.
    Now, for each $\omega\in\Omega$, we use \Cref{th:proof_approx_box} to get a uniform approximation
    $Z_n(\cdot; \omega)\rightarrow Z(\cdot; \omega)$ as in \Cref{eq:proof_approx_box}. We then have:
    \begin{align*}
        \ell\circ\Phi(\omega)&=\ell\left(\lim_{n\rightarrow\infty}Z_n\left(\cdot;\omega\right)\right)
        = \lim_{n\rightarrow\infty}\int
        \sum_{k=0}^{K(n)}Z\left(t^{(n)}_{k};
        \omega\right)\mathbb{1}_{A^{(n)}_{k}}d\lambda\\
                          &= \lim_{n\rightarrow\infty}\sum_{k=0}^{K(n)}
                          Z\left(t^{(n)}_{k};
                          \omega\right)\lambda\left(A^{(n)}_{k}\right).
    \end{align*}
    Now, as a convergent series of Gaussian random variables, the above is
    Gaussian (use characteristic functions and L\'{e}vy convergence theorem).
\end{proof}

We now turn to the proof of \Cref{th:gaussianity_induced_process}.

\begin{proof}{(\Cref{th:gaussianity_induced_process})}
    Let $\Omega=C(D)$ and $\mathcal{F}$ be the Borel sigma algebra on $C(D)$
    and define a collection of random variables
    \[
        Z_s:\left(\Omega,\mathcal{F}, \mu\right)\rightarrow
            \left(\mathbb{R}, \mathcal{B}\left(\mathbb{R}\right)\right),~ \omega\mapsto
            \delta_s\left(\omega\right)
    \]
    for all $s \in D$. Since for all $s\in D$, the Dirac functionals
    $\delta_s$ belong to the dual of $C(D)$, we have that $Z_t$ is a Gaussian
    real random variable for all $s$. Now, for $s_1, ..., s_n \in D$, any
    linear combination of the components of the vector $\left(Z_{s_1}, ..., Z_{s_n}\right)$
    may be written an element of $C(D)^*$, and will hence be Gaussian
    distributed by Gaussianity of the measure. This shows that $Z$ is a
    Gaussian process on $D$.
\end{proof}
From the above theorems, it is also clear that if $Z$ is the process induced by a Gaussian measure on $C(D)$, then for any $s\in D$, we have
\begin{equation}\label{eq:corresp_expectations}
    \mathbb{E}_{\mu}\left[f(s)\right] = \mathbb{E}\left[Z_s\right]
\end{equation}
and the same is true if $Z$ is a GP on $D$ with trajectories in $C(D)$ and $\mu$ is the measure induced by the process. This allows us to translate everything from process to measure and back without needing to worry about the details. Finally, using the fact that the Dirac deltas belong to the dual we may also prove \Cref{th:corresp_mean_cov} about the correspondence between mean element and covariance operator of the induced measure and mean and covariance function of the process.
\begin{proof}{(\Cref{th:corresp_mean_cov})}
For $s, s_1, s_2\in D$, let:
 \begin{align*}
     m_s :&= \langle m_{\mu}, \delta_s\rangle = \int_{C(D)}f(s)d\mu(f) = \mathbb{E}_{\mu}\left[f(s)\right]=\mathbb{E}\left[Z_s\right]\\
     k(s_1, s_2) :&= \langle \delta_{s_1}, C_{\mu}\delta_{s_2} \rangle = \delta_{s_1}\left[\int_{C(D)} f(s_2)f d\mu(f) - m_{s_1} m_{\mu}\right]\\
      &= \mathbb{E}_{\mu}\left[f(s_1)f(s_2)\right] - \mathbb{E}_{\mu}\left[f(s_1)\right]\mathbb{E}\left[f(s_2)\right]\\
      &= \mathbb{E}\left[Z_{s_1}Z_{s_2}\right] - \mathbb{E}\left[Z_{s_1}\right]\mathbb{E}\left[Z_{s_2}\right].
 \end{align*}
Note that exchanging Dirac deltas and integration is allowed by Fubini since Gaussian measures are finite and the last equalities are consequences of \Cref{eq:corresp_expectations}.
\end{proof}
The extension of \Cref{th:gaussianity_induced_measure} and \Cref{th:gaussianity_induced_process} 
to processes and measures on RKHS is straightforward. Indeed, 
the measure-to-process correspondence follows directly from the fact 
that the evaluation functionals 
belong to the dual of the RKHS. 
For the process-to-measure correspondence, the crucial property is 
the Gaussianity of 
linear functionals of the field, which in a RKHS $\mathcal{H}$ is 
automatically satisfied since any linear 
functional can be expressed as an infinite linear combination of 
reproducing kernel values, which in turn act as evaluation functionals: 
\begin{align*}
    \left\langle \ell, Z \right\rangle &= 
    \left\langle \sum_{i=1}^{\infty} a_i k(x_i, \cdot), Z\right\rangle_{\mathcal{H}} = 
    \sum_{i=1}^{\infty} a_i Z_{x_i},
\end{align*}
which, as a convergent sum of Gaussian random variables, is Gaussian.

\section{Conditioning, Disintegration and Link to Finite-Dimensional Formulation:}\label{sec:math_and_proofs}
We now turn to the proof of \Cref{th:tarieladze_disintegration}.

\begin{proof}{(\Cref{th:tarieladze_disintegration})}
    To prove the theorem, we have to adapt the proof of \citet{TARIELADZE2007851}[Theorem 3.11] to the non-centered case. Compared to the original theorem, the conditional covariance operator $\tilde{C}_{\mu}$ hasn't changed, whereas the conditional mean $\tilde{m}_{\mu}\left(y\right)$ clearly still defines a continuous mapping satisfying $G\tilde{m}_{\mu}\left(y\right)=y$ for all $y$ in the range of $C_{\nu}$.
    Hence, for all $y \in Y$, we can still use \citet{TARIELADZE2007851}[Lemma 3.8] to define $\mu_{|G=y}$ as a Gaussian measure having mean element $\tilde{m}_{\mu}\left(y\right)$ and covariance operator $\tilde{C}_{\mu}$. What is left to check is that it satisfies the conditions in \Cref{def:disintegration} to be a disintegration of $\mu$ with respect to $G$.

In the following, let $y \in Y$ and $A\in\mathcal{A}$ be arbitrary.
\begin{itemize}
    \item The measurability of the mapping $y \mapsto \mu_{|G=y}\left(A\right)$ for fixed $A$ holds since, compared to the centered case, the conditional mean $\tilde{m}_{\mu}\left(y\right)$ is only translated by an element that does not depend on $y$.
    \item Define $Y_0:=G m_{\mu} + C_{\nu}\left(Y^*\right)$. We have $\mu\circ G^{-1}\left(Y_0\right)=1$ by \citet{TARIELADZE2007851}[Lemma 3.3] and \citet{TARIELADZE2007851}[Corollary 3.7]. Following the exact same reasoning as in the proof of \citet{TARIELADZE2007851}[Theorem 3.11] we have that $\mu_y\left(G^{-1}\left(y\right)\right)=1$.
    \item By \citet{TARIELADZE2007851}[Proposition 3.2], the last thing we have to check is that 
    \[
    \hat{\mu}\left(x^*\right) = \int_Y \hat{\mu}_{|G=y}\left(x^*\right)d\nu\left(y\right),~\forall x^* \in X^*,
    \]
    where $\hat{\mu}\left(\boldsymbol{\cdot}\right)$ denotes the characteristic functional of $\mu$ (see \citet{TARIELADZE2007851}[Section 3.2]. Compared to the original proof, only the mean element is changed, so for the sake of simplicity we only consider the steps of the proof that differ from the original ones.\\
    We have that 
    \begin{align*}
        \int_Y \exp\left[i\left\langle \tilde{m}_{\mu}\left(y\right), x^*\right\rangle\right]d\nu\left(y\right) &= 
        \exp\left[i\left\langle m_{\mu}, x^* \right\rangle\right]\\
                                                                                                                & \cdot \int_y \exp\left[i \left\langle
        \sum_{i=1}^n \left\langle y - G m_{\mu}, y_i^* \right\rangle C_{\mu} G* y_i^*, x^*
        \right\rangle\right]d\nu\left(y\right),
    \end{align*}
    which, after a change of variable $y\mapsto y - Gm_{\mu}$ can be seen to be the characteristic function of a centered Gaussian measure with covariance $C_{\mu}$ by following the same argument as in the original proof (the same argument is presented in more detail in the proof of the next 
    theorem).
    \end{itemize}
\end{proof}
The last theorem we have to prove is the one about the transitivity of disintegrations.
\begin{proof}{(\Cref{th:transitivity_disintegration})}
By unicity of disintegrations \citep{TARIELADZE2007851}[Remark 3.12], we only have to prove that the family 
\[
    \left(\left(\mu_{|G_1=y_1}\right)_{|G_2=y_2}\right)_{\left(y_1, y_2\right)\in Y_1 \bigoplus Y_2}
\] defines a disintegration of $\mu$ with respect to $\left(G_1, G_2\right)$.\\
First a word of caution: there exist no \textit{canonical} norm on the direct sum of Banach spaces. However, there are several norms on the direct sum that induce the product topology \citep{salamon}[Exercice 1.30]. We here work with any of these. Then, the Borel $\sigma$-algebra on the direct sum is given by the product of the Borel $\sigma$-algebras of the components \citep{billingsley}[p.244].

Here, by construction, for any $\left(y_1, y_2\right) \in Y_1 \bigoplus Y_2$, the measure $\left(\mu_{| G_1=y_1}\right)_{|G_2=y_2}$ is defined as a Gaussian measure having mean element
\begin{align*}
    \tilde{m}_{\mu}^{(1)}\left(y_1\right) + \sum_{i=1}^{\datdim_2} \left\langle
    y_2 - G_2\tilde{m}_{\mu}^{(1)}\left(y_1\right), y_i^{(2)*}\right\rangle \tilde{C}_{\mu}^{(1)}G_2^* y_i^{(2)*},
\end{align*}
and covariance operator
\begin{align*}
	\tilde{C}_{\mu}^{(1,2)} :&= \tilde{C}_{\mu}^{(1)} - \sum_{i=1}^{\datdim_2}\left\langle \tilde{C}_{\mu}^{(1)} G_2^* y_i^{(2)*}, \cdot\right\rangle
	\tilde{C}_{\mu}^{(1)} G_2^* y_i^{(2)*},
\end{align*}
where $\tilde{m}_{\mu}^{(1)}$ and $\tilde{C}_{\mu}^{(1)}$ denote the mean element 
and covariance operator of $\mu_{| G_1=y_1}$ and 
$(y_i^{(2)*})_{i=1, ..., \datdim_2}$ is any representing sequence 
for the operator $G_2 \tilde{C}_{\mu}^{(1)} G_2^*$. Note that the assumption that 
$C_{\nu_2}$ has finite rank implies that the aforementioned operator also has finite rank. 
Since for all $y_1\in Y_1$ the measure $\mu_{| G_1=y_1}$ is Gaussian, we have by \Cref{th:tarieladze_disintegration} that $\left(\mu_{| G_1=y_1}\right)_{| G_2 = y_2}$ is Gaussian.

As in the previous proof, we have to check the three conditions of \Cref{def:disintegration}.
\begin{itemize}
	\item For fixed $A$, the mapping $\left(y_1, y_2\right) \mapsto 
\left(\mu_{|G_1=y_1}\right)_{G_2=y_2}\left(A\right)$ is an addition   of a $\mathcal{B}\left(Y_1\right)$-measurable mapping with a $\mathcal{B}\left(Y_2\right)$-measurable mapping, and, as such, measurable with respect to the product $\sigma$-algebra.

    \item Let $Y:=Y_1 \bigoplus Y_2$ and note that $Y^* = Y_1^* \bigoplus Y_2^*$ (dual of direct sum is the direct sum of the duals). Then define 
    $Y_0 = Gm_{\mu} + G C_{\mu} G^{*}\left(Y_1^* \bigoplus Y_2^*\right)$. Note that the Gaussian measure $\mu\circ G^{-1}$ has mean $G m_{\mu}$ and covariance operator $G C_{\mu} G^*$, hence 
    $\mu \circ G^{-1}\left(Y_0\right) = 1$ 
    by \citet{TARIELADZE2007851}[Lemma 3.3].
    
    For any $\left(y_1, y_2\right) \in Y_0$ we have that the Gaussian measure $\left(\mu_{| G_1 = y_1}\right)_{| G_2 = y_2} \circ G^{-1}$ has covariance operator $G \tilde{C}_{\mu}^{(1,2)} G^*$. Computing the operator componentwise, we have that:
    \begin{align*}
    ´	G_2 \tilde{C}_{\mu}^{(1,2)} G_2^* &= G_2 \tilde{C}_{\mu}^{(1)} G_2^*  - \sum_{i=1}^{\datdim_2} \left\langle
    \tilde{C}_{\mu}^{(1)}G_2^* y_i^{(2)*}, G_2^* \right\rangle 
    G_2 \tilde{C}_{\mu}^{(1)} G_2^* y_i^{(2)*} = 0,
    \end{align*}
    where the last equality follows from \citet{TARIELADZE2007851}[Lemma 3.4, (c)] since $y_i^{(2)*}$ is a $G_2\tilde{C}_{\mu}^{(1)}G_2^*$-representing sequence. An analogous computation for the other components shows that they all vanish.
    \end{itemize}

Defining $\mu^{(1,2)}_{y_1, y_2}:=\left(\mu_{| G_1=y_1}\right)_{| G_2 = y_2}$, the last point to show is that
    \begin{align*}
    	\hat{\mu}\left(x^*\right) &= 
    	\int_{Y_1}\int_{Y_2} \hat{\mu}^{(1,2)}_{y_1, y_2}\left(x^*\right)d\nu_1\left(y_1\right)d\nu_2\left(y_2\right),~\forall x^* \in C(D)^*,
    \end{align*}
    where we have defined the measures $\nu_1 := \mu \circ G_1^{-1}$ and $\nu_2 := \mu_{| G_1 = y_1}\circ G_2^{-1}$, omitting the dependence on $y_1$ for simplicity. Using the fact that $\mu_{| G_1=y_1}$ is a disintegration of $\mu$ with respect to $G_1$, defining $x_j^{(2)*}:= G_2^* y^{(2)*}$ and performing a change of variables, we have that
    \begin{align*}
    \int_{Y_1}\int_{Y_2} \hat{\mu}^{(1,2)}_{y_1, y_2}\left(x^*\right)d\nu_1\left(y_1\right)d\nu_2\left(y_2\right)
    &= \hat{\mu}\left(x^*\right)
    + \int_{Y_1}\int_{Y_2} \exp\left[i\sum_j 
    \left\langle y_2, y_j^{(2)*}\right\rangle \left\langle C_{\mu} x_j^{(2)*}, x^*\right\rangle \right. \\
    &+ \left. \frac{1}{2} \sum_j \left\langle C_{\mu}^{(1)}x_j^{(2)*}, x^* \right\rangle \left\langle C_{\mu}^{(1)} x_j^{(2)*}, x^*\right\rangle\right].
    \end{align*}
Now defining $M\left(y_2\right):= \sum_j \left\langle y_2, y_j^{(2)*}\right\rangle C_{\mu}^{(1)}x_j^{(2)}$ and $R_1 := \sum_j \left\langle C_{\mu}^{(1)} x_j^{(2)*}, \cdot\right\rangle C_{\mu}^{(1)}x_j^{(2)*}$ (which can be thought of as conditional means and covariance operator), the remaining integral reduces to:
\begin{align*}
\int_{Y_1} \int_{Y_2} \exp \left[
    i \left\langle M\left(y_2\right), x^*\right\rangle + \frac{1}{2} R_1 x^*\right]d\nu_1\left(y_1\right) d\nu_2\left(y_2\right).
    \end{align*}
    We can simplify the first summand by noticing that it amounts to the characteristic function of a Gaussian measure:
    \begin{align*}
    \int_{Y_2} \exp \left[
    i \left\langle M\left(y_2\right), x^*\right\rangle\right] d\nu_2\left(y_2\right)
    &= \hat{\nu}_2\left(M^*\left(x^*\right)\right) 
    = \int_{Y_2} \exp \left[-\frac{1}{2} \left\langle C_{\nu_2} M^*\left(x^*\right), M^*\left(x^*\right)\right\rangle\right]\\
    &= \int_{Y_2} \exp \left[-\frac{1}{2} \sum_j\left\langle
    C_{\mu}^{(1)}x_j^{(2)*}, x^*\right\rangle^2\right] d\nu_2\left(y_2\right)\\
    &= 
    \int_{Y_2} \exp \left[-\frac{1}{2} R_1 x^*\right]d\nu_2\left(y_2\right),
    \end{align*}
    where the penultimate equality follows from $C_{\nu_2}$-orthogonality of the $y_i^{(2)*}$ sequence. This concludes the proof. Note when computing the characteristic function of $\nu_2$, we have omitted the mean term due to the change of variable performed earlier (to be perfectly rigorous, one should use a different notation for the transformed measure).
\end{proof}

\textbf{Link to Finite Dimensional case}
When the inversion data is
\textit{finite-dimensional}, that is the observation operator $G$ maps into
$\mathbb{R}^n$ and $\mathbb{R}^n$ is considered as a Banach space with respect
to the $2$-norm. One can then canonically identify $\mathbb{R}^n$ with its dual
using the dot product: $v\mapsto \langle v, \cdot \rangle$. In the
following, when elements of $\mathbb{R}^n$ are involved, the duality bracket
$\langle \cdot, \cdot \rangle$ will denote the dot product, also, $e_i, i=1,
..., n$ will be used to denote the canonical basis of $\mathbb{R}^n$.
We now prove that $y_i:=C_{\nu}^{-1/2}e_i,~i=1, ..., n$ forms a $C_{\nu}$-representing sequence.
\begin{proof}{(\Cref{th:remark_rep_seq})}
    First of all, the $y_i$ form a $C_{\nu}$-orthonormal family since
    \begin{align*}
        \langle C_{\nu} y_i, y_j \rangle &= \langle C_{\nu}^{1/2} y_i, C_{\nu}^{1/2} y_j \rangle
        = \langle e_i, e_j \rangle = \delta_{ij},
    \end{align*}
    where the first equality follows by self-adjointness of $C_{\nu}$. Also
    remember that since here we are working over $\mathbb{R}^n$, the duality
    bracket denotes the dot product and $\mathbb{R}^n$ is identified with its
    dual. Finally, according to \citet{TARIELADZE2007851}[Lemma 3.4], the last
    thing we have to show is 
    that for any $v \in \mathbb{R}^n$: $C_{\nu}v = \sum_{i=1}^n
    \langle C_{\nu} y_i, v\rangle C_{\nu} y_i$. Note that since $C_{\nu}$ is a
    positive self-adjoint operator, the $y_i$'s form a basis of $\mathbb{R}^n$,
    and we can thus write $v=\sum_{i=1}^n v_i y_i$ for some component $v_i$. Then 
    \begin{align*}
        \sum_{i=1}^n \langle C_{\nu} y_i, v\rangle C_{\nu} y_i &= \sum_{i,j = 1}^n \langle
        C_{\nu} y_i, v_j y_j \rangle C_{\nu} y_i = v^i C_{\nu}y_i = C_{\nu} v
    \end{align*}
\end{proof}

\begin{proof}{(\Cref{th:link_tarantola})}
As before, let $y_i:=C_{\nu}^{-1/2}e_i,~i=1, ..., n$.
In order to get closed-form formulae for the posterior under such operators, we
need to be able to compute the action of the adjoint $G^*$. We begin by
recalling the
definition of the adjoint of a linear operator $T:X\rightarrow Y$ between
Banach spaces:
\begin{align*}
    T^*:&Y^*\rightarrow X^*\\
        &y^*\mapsto \left(x\mapsto \langle y^*, Tx\rangle\right).
\end{align*}


Now if we consider a (bounded) linear form $G_j:X:\rightarrow
\mathbb{R}$, then its adjoint is given by:
\begin{align*}
    G_j^*:&\mathbb{R}\rightarrow X^*\\
    &a\mapsto \left(f\mapsto a \cdot G_j f \right).
\end{align*}
So the adjoint of the observation operator may be written as:
\begin{align*}
    G^*:\mathbb{R}^n &\rightarrow X^*\\
    \left(a_1, ..., a_n\right) &\mapsto \left( f\mapsto
            a_1 \cdot G_1 f + ... + a_n
    \cdot G_n f \right).
\end{align*}

There is one last computation that we need to perform before getting the mean
and covariance:
\begin{align*}
    \left\langle C_{\mu}G^* y^{(i)}, \delta_s \right\rangle 
    &= \left\langle C_{\mu}\delta_s, G^* y^{(i)} \right\rangle 
    = y^{(i)} \cdot G\left(C_{\mu}\delta_s\right)= y^{(i)} \cdot G k(\cdot, s)
    = y^{(i)}\cdot K_{sG}.
\end{align*}
~\\
Putting everything together we are now able to express the covariance
operator:
\begin{align*}
    \tilde{k}(s_1, s_2) &= k(s_1,s_2) - \sum_{i=1}^n y^{(i)}\cdot K_{s_1G} y^{(i)}\cdot K_{s_2G}\\
                    &= k(s_1, s_2) - \sum_{i=1}^n K_{s_1G}^T y^{(i)} \left(y^{(i)}\right)^T K_{s_2G}
                    \\
                    &= k(s_1, s_2) - \sum_{i=1}^n K_{s_1G}^T C_{\nu}^{-1/2}e_i e_i^T C_{\nu}^{-1/2} K_{s_2G}\\
                    &= k(s_1, s_2) - K_{s_1G}^T K_{GG}^{-1} K_{s_2G}.
\end{align*}
Where we have used the fact that $\sum_{i=1}^n e_i e_i^T =\bm{I}_n$ and that:
\begin{align*}
    e_i\cdot GC_{\mu}G^{*}e_j &= G_i(G_j k(\cdot, \cdot)).
\end{align*}
Note that this last step requires one to explicitly compute the action of the $G_i$'s 
on the covariance operator $C_{\mu}$. This can be done in the case where $X=C(D)$ since 
the individual components on the observation operator can be written as integrals 
with respect to Radon measures $G_i f = \int_D f(s) d\lambda_i\left(s\right)$ or 
in the case where $X$ is a RKHS, since then the components can be written 
as infinite linear combinations of Dirace deltas $G_i f = \sum_{k=1}^{\infty} a_k^{(i)}f(s^{(i)}_k)$. Computing the action on the covariance operator in the general case is not trivial.
The mean can be obtained through a similar argument.
\end{proof}

\textbf{Proofs for Infinite Rank Data}
\begin{proof}{(\Cref{th:tarieladze_disintegration_infinite})}
As before, compared to the centered case, only the conditional mean changes. 
Thanks to \cite[Lemma 3.5]{TARIELADZE2007851} we can still 
select a countably 
infinite $C_{\nu}$ representing sequence $\left(y_i\right)_{i \in \mathbb{N}}$.
Now define, for all $n \in \mathbb{N}$:
    \begin{align}
        \tilde{m}^{(n)}_{\mu}(y) &= m_{\mu}
        +
        \sum_{i=1}^{n} \left\langle y - Gm_{\mu}, y_i^* \right\rangle C_{\mu} G^* y_i^*.
    \end{align} 
Furthermore, define the spaces: 
$Y_2:=\lbrace y \in Y:~\tilde{m}^{(n)}_{\mu}(y) \text{ converges}\rbrace$, 
and $Y_3:= \lbrace y\in Y:~\lim_{n\rightarrow \infty} || y 
- \sum_{i=1}^n \left\langle y - G m_{\mu}, y_i^*\right\rangle C_{\nu}y_i^* || = 0\rbrace$. 
    We begin by showing that these subspaces of $Y$ have full measure.\\

\textbf{Claim:} $\nu(Y_2) = 1$.
\begin{proof}
    Our goal is to show that the random element 
    $\tilde{m}_{\mu}^n$ converges $\nu$-almost surely in $X$. 
    First, define $\xi_i:=\left\langle y - Gm_{\mu}, y_i^* \right\rangle C_{\mu} G^* y_i^*$. 
    Thanks to $C_{\nu}$-orthonormality, 
    the $y_i^*$ are independent Gaussian random variables, and hence the $\xi_i$ too.
    Hence, 
    by Ito-Nisio \cite[Theorem 5.2.4]{Vakhania1987}, 
    we get $\nu$-almost-sure convergence provided we can show that 
    there exists a random probability measure $\mu'$ on $X$ such that 
    the joint characteristic function converges to the characteristic function of $\mu'$:
    \begin{align*}
        \prod_{i=1}^n \hat{\mathbb{P}}_{\xi_i}
        \left(f\right)\rightarrow \hat{\mu'}(f),~\text{all }f\in X^*.
    \end{align*}
    By independence of the $\xi_i$, we have, for $f \in X^*$:
    \begin{align*}
        \prod_{i=1}^n \hat{\mathbb{P}}_{\xi_i}(f) &=
        \int_Y \exp\left[
            i \left\langle
                f, \sum_{i=1}^n\left\langle y - Gm_{\mu}, y_i^*\right\rangle
                C_{\mu}G^*y_i^*\right\rangle
            \right]d\nu(y)\\
                                                  &=
        \int_Y \exp\left[
            i \left\langle
                y', \sum_{i=1}^n y_i^*\left\langle
                    f, C_{\mu}G^*y_i^*\right\rangle
                    \right\rangle
                \right] d\nu'(y'),
    \end{align*}
    where we have performed a change of variable $y':= y - Gm_{\mu}$ and hence $\nu$' 
    is a centred Gaussian measure with covariance operator $C_{\nu}$. Now, using the 
    characteristic function of Gaussian measures, the above is equal to:
    \begin{align*}
        \hat{\nu}\left(\sum_{i=1}^n y_i^*\left\langle 
        f, C_{\mu}G^*y_i^*\right\rangle\right)
        &=
        \exp\left[
            -\frac{1}{2}\sum_{i, j=1}^n\left\langle 
                f, C_{\mu}G^*y_i^*\right\rangle
                \left\langle f, C_{\mu}G^* y_j^*\right\rangle
            \left\langle C_{\nu}y_i^*, y_j^*\right\rangle\right]\\
        &=
        \exp\left[-\frac{1}{2}\sum_{i=1}^n \left\langle
        f, C_{\mu}G^*y_i^*\right\rangle^2\right],
    \end{align*}
    where the last equality follows from $C_{\nu}$-orthonormality of the 
    representing sequence. We thus have:
    \begin{align}
        \lim_{n\rightarrow \infty} \prod_{i=1}^n \hat{\mathbb{P}}_{\xi_i}(f) 
        &= \exp\left[-\frac{1}{2}\left\langle R_1 f, f\right\rangle\right]\label{eq:prod_charac},
    \end{align}
    where $R_1:=\lim_{n\rightarrow \infty}\sum_{i=1}^{\infty}
    \left\langle C_{\mu}G^*y_i^*, \bullet\right\rangle
    C_{\mu}G^*y_i^*$ is a Gaussian covariance by \cite[Lemma 3.4 and Proposition 3.9]{TARIELADZE2007851}. 
    The Claim follows from the fact that for any Gaussian covariance, there 
    exists a Gaussian measure having that covariance as covariance operator 
    \cite[Lemma 3.8]{TARIELADZE2007851}.
\end{proof}

\textbf{Claim:} $\nu(Y_3) = 1$. 
\begin{proof}
    Note that if $y - G m_{\mu}$ can be written as $C_{\nu}y^*$ for some $y^*\in Y^*$, 
    then it immediatly follows, by \cite[Lemma 3.4]{TARIELADZE2007851}, that:
    \begin{align*}
        \sum_{i=1}^{\infty}\left\langle y - G m_{\mu}, y_i^*\right\rangle C_{\nu}y_i^* 
        &= 
        \sum_{i=1}^{\infty}\left\langle y^*, C_{\nu}y_i^*\right\rangle C_{\nu}y_i^* 
        = C_{\nu}y^* = y - Gm_{\mu}.
    \end{align*}
    Now, the subspace whose elements can be written as above is exatly the 
    Cameron-Martin space $C_{\nu}\left(Y^*\right)$. While this is a $\nu$-null space, 
    it is a well-known fact that its closure in $Y$ has full measure, 
    so that there exists a subset of full measure whose elements can 
    be approximated by elements of $C_{\nu}\left(Y^*\right)$ and thus 
    the defining property of $Y_3$ holds on a set of full measure.
\end{proof}

Now, we define $Y_0:=Y_2\cap Y_3$. 
We construct a disintegration $\left(\mu_{|G=y}\right)_{y\in Y_0}$ as in the finite 
rank case, but now restricting to the subspace $Y_0$ where the conditional mean 
is defined. What is left to check is that it satisfies the three defining 
properties of disintegrations \pcref{def:disintegration}. Property 1 holds as in 
the finite rank case. For Property 2, we notice that, for any $y \in Y_0$:
\begin{align*}
    G\tilde{m}_{\mu}(y) &= \lim_{n\rightarrow\infty} \tilde{m}_{\mu}^{(n)}(y)
    = G m_{\mu} - \sum_{i=1}^{\infty}\left\langle G m_{\mu}, y_i^*\right\rangle
    C_{\nu}y_i^*
    + \sum_{i=1}^{\infty}\left\langle 
    y, y_i^*\right\rangle C_{\nu} y_i^* = y,
\end{align*}
since $G m_{\mu}$ is the mean of $\nu$ and thus belongs to the Cameron-Martin space. 
Finally, for Property 3, 
thanks to \cite[Proposition 3.2]{TARIELADZE2007851}, we only 
have to show that the characteristic function of $\mu$ writes as a mixing 
of the characteristic functions of the conditionals, i.e. that:
\begin{align*}
    \hat{\mu}(f) &= \int_Y \hat{\mu}_{|G=y}(f)d\nu(y),\text{ all } f\in X^*.
\end{align*}
Now, for $y \in Y_0$, we have that $\mu_{|G=y}$ is Gaussian, with mean 
$\tilde{m}_{\mu}(y)$ and covariance operator $C_{\mu} - R_1$. Hence, we have:
\begin{align*}
    \int_Y \hat{\mu}_{|G=y}(f)d\nu(y) &=
    \int_Y \exp\left[
        i\left\langle \tilde{m}_{\mu}(y), f\right\rangle 
        - \frac{1}{2}\left\langle C_{\mu}, f\right\rangle 
    + \frac{1}{2}\left\langle R_1 f, f\right\rangle\right]d\nu(y)\\
                                      &=
    \exp\left[i\left\langle m_{\mu}, f\right\rangle - \frac{1}{2}
    \left\langle C_{\mu}f, f\right\rangle\right] 
= \hat{\mu}(f),
\end{align*}
where the second-to-last equality follow from \Cref{eq:prod_charac}. This completes 
the proof in the infinite rank case.
\end{proof}

\section{Explicit Update Formulae for Mean Element and Covariance Operator}\label{sec:explicit}
For the sake of completeness, we here provide detailed update formulae for the mean element and covariance operator, 
as a direct consequence of \Cref{th:transitivity_disintegration}.

\begin{corollary}
    Consider the setting of \Cref{th:transitivity_disintegration} and 
    let $(y_i^{*(12)})_{i=1,...,p_{12}}$ be a $G C_{\mu} G^*$-representing 
sequence, $(y_i^{*(1)})_{i=1,...,p_{1}}$ be a $G_1 C_{\mu} G_1^*$-representing 
sequence and $(y_i^{*(2)})_{i=1,...,p_{2}}$ be a 
    $G_2 C^{(1)}_{\mu} G_2^*$-representing sequence. Then we have:
    {\small
\begin{align*}
    C_{\mu} - \sum_{i=1}^{p_{12}}\left\langle C_{\mu}G^*y_i^{(12)}, \bullet\right\rangle
    C_{\mu}G^*y_i^{*(12)}
    &=
    C_{\mu}
    - \sum_{i=1}^{p_1}\left\langle C_{\mu}G^*y_i^{*(1)}, \bullet\right\rangle
    C_{\mu}G_1^*y_i^{*(1)}
    -
    \sum_{j=1}^{p_2}\left\langle C_{\mu}G_2^* y_j^{*(2)}, \bullet\right\rangle
    C_{\mu}G_2^*y_j^{*(2)}\\
    +
    \sum_{j=1}^{p_2}\sum_{i=1}^{p_1}& \left\langle C_{\mu}G_2^* y_j^{*(2)}, \bullet\right\rangle
    \left\langle C_{\mu}G_1^* y_i^{*(1)}, G_2^*y_j^{*(2)}\right\rangle
    C_{\mu}G^*y_i^{*(1)}\\
    -
    \sum_{j=1}^{p_2}\sum_{i=1}^{p_1}\sum_{k=1}^{p_1}&
    \left\langle C_{\mu}G_1^*y_i^{*(1)}, G_2^* y_j^{*(2)}\right\rangle
    \left\langle C_{\mu}G_1^* y_i^{*(1)}, \bullet\right\rangle
    \left\langle C_{\mu} G_1^*y_i^{*(1)}, G_2^*y_j^{*(2)}\right\rangle
    C_{\mu}G_1^*y_i^{*(1)},
\end{align*}
}
and the equality is independent of the choice of the representing sequences. 

As for the mean element, we have:
     \begin{align*}
         m_{\mu} 
         + 
         \sum_{i=1}^{\datdim} 
         \left\langle y - Gm_{\mu}, y_i^{*(12)} \right\rangle C_{\mu} G^* y_i^{*(12)}
         =
         m_{\mu} 
         + 
         \sum_{i=1}^{\datdim_1} 
         \left\langle y_1 - G_1 m_{\mu}, y_i^{*(1)} \right\rangle C_{\mu} G_1^* y_i^{*(1)}&
         \\
         +
         \sum_{j=1}^{\datdim_2}
         \left\langle
         y_2, y_j^{*(2)}\right\rangle
         C_{\mu}G_2^* y_j^{*(2)}
         -
         \sum_{j=1}^{\datdim_2}
         \sum_{i=1}^{\datdim_1}
         \left\langle 
             y_2, y_j^{*(2)}
         \right\rangle 
         \left\langle
                 C_{\mu}G_1^* y_i^{*(1)}, G_2^* y_j^{*(2)}
         \right\rangle 
                 C_{\mu}G_1^* y_i^{*(1)}\\
         - 
         \sum_{j=1}^{\datdim_2} 
         \left\langle G_2m_{\mu}, y_j^{*(2)} \right\rangle 
         C_{\mu} G_2^* y_j^{*(2)}\\
         -
         \sum_{j=1}^{\datdim_2} 
         \sum_{i=1}^{\datdim_1} 
         \left\langle G_2 C_{\mu}G_1^*y_i^{*(1)}, y_j^{*(2)} \right\rangle 
         \left\langle y_1 - G_1 m_{\mu}, y_i^{*(1)}\right\rangle
         C_{\mu} G_2^* y_j^{*(2)}\\
         + 
         \sum_{j=1}^{\datdim_2} 
         \sum_{k=1}^{\datdim_1}
         \left\langle G_2m_{\mu}, y_j^{*(2)} \right\rangle 
         \left\langle
             C_{\mu}G_1^*y_k^{*(1)}, G_2^* y_j^{*(2)}
             \right\rangle
             C_{\mu}G_1^*y_k^{*(1)}
             \\
         +
         \sum_{i=1}^{\datdim_2} 
         \sum_{i=j}^{\datdim_1} 
         \sum_{k=1}^{\datdim_1}
         \left\langle G_2 C_{\mu}G_1^*y_i^{*(1)}, y_j^{*(2)} \right\rangle 
         \left\langle y_1 - G_1 m_{\mu}, y_i^{*(1)}\right\rangle
         \left\langle
             C_{\mu}G_1^*y_k^{*(1)}, G_2^* y_j^{*(2)}
             \right\rangle
             C_{\mu}G_1^*y_k^{*(1)}.
     \end{align*}
 \end{corollary}



\bibliography{bibliography}

\end{document}